\documentclass[12pt,a4paper,twoside,flegn]{article}
\usepackage{amsmath}
\usepackage{amsfonts}
\usepackage{amssymb}
\usepackage{amsthm}
\usepackage{enumerate}
\usepackage{amscd}
\usepackage{amsrefs}
\usepackage{hyperref}
\usepackage{verbatim}
\usepackage{calligra}
\usepackage{mathrsfs}
\usepackage{xcolor}
\usepackage[all,cmtip]{xy}
\usepackage[upright]{fourier}

\fontfamily{Liberation Serif}

\hypersetup{citebordercolor={0 1 0}}
\newcommand{\Et}{\text{E}}
\newcommand{\Ht}{\text{H}}

\newcommand{\ts}{\text{Spec}}

\newcommand{\hit}{\text{ht}}
\newcommand{\grd}{\text{grade}}
\newcommand{\As}{\text{Ass}}
\newcommand{\At}{\text{Att}}
\newcommand{\as}{\text{ass}}
\newcommand{\dep}{\text{Depth}}
\newcommand{\dime}{\text{dim}}

\newcommand{\map}{\mathfrak{p}}
\newcommand{\mam}{\mathfrak{m}}
\newcommand{\maq}{\mathfrak{q}}
\newcommand{\maa}{\mathfrak{a}}
\newcommand{\mab}{\mathfrak{b}}
\newcommand{\mac}{\mathfrak{c}}

\newcommand{\man}{\mathfrak{n}}

\newcommand{\mn}{\mathbb{N}}

\newcommand{\Supp}{\text{Supp}}

\newcommand{\Var}{\text{Var}}
\newcommand{\Max}{\text{Max}}

\newcommand{\pd}{\text{Pd}}

\newcommand{\ext}{\text{Ext}}
\newcommand{\homm}{\text{Hom}}

\newcommand{\dsuml}{\bigoplus\limits}
\newcommand{\dsum}{\bigoplus}

\newcommand{\ten}{\bigotimes}
\newcommand{\tenr}{\bigotimes_R}
\newcommand{\insl}{\bigcap\limits}
\newcommand{\ins}{\bigcap}
\newcommand{\uni}{\bigcup}
\newcommand{\unil}{\bigcup\limits}

\newcommand{\Rc}{\widehat{R}}
\newcommand{\Mc}{\widehat{M}}

\DeclareMathAlphabet{\mathcalligra}{T1}{calligra}{g}{f}

\begin{document}
   
    \date{}
    \title{\Large \textbf{A Study of Quasi-Gorenstein Rings}}
    \markboth{\small \textsc{A Study of Quasi-Gorenstein Rings}}{\small  \textsc{Ehsan Tavanfar and Massoud Tousi}}         
      \author{\normalsize \textsc{EHSAN TAVANFAR} \normalsize and \textsc{MASSOUD TOUSI}}
    \maketitle          
    \begin{abstract}
       \textsc{Abstract}. In this paper  several quasi-Gorenstein counterparts to some known properties of Gorenstein rings are given. We, furthermore,  give an explicit description of the attached prime ideals of certain   local cohomology modules.

    \end{abstract}     
    
    {\let\thefootnote\relax\footnotetext{2010 Mathematics Subject Classification. 13H10, 13D45.}}
    {\let\thefootnote\relax\footnotetext{Key words and phrases. Canonical module, $G$-dimension,  Gorenstein rings, limit closure, local cohomology, quasi-Gorenstein rings.}}
    {\let\thefootnote\relax\footnotetext{This research was in part supported by a grant from IPM (No. 93130211).}}
    
    \newtheorem{thm}{Theorem}[section]
    \theoremstyle{Definition}    
    \newtheorem{defi}[thm]{Definition}
    \theoremstyle{Definition and Remark}
    \newtheorem{defi-rem}[thm]{Definition and Remark}
    \newtheorem{defis-rem}[thm]{Definitions and Remark}
    \newtheorem{defis-rems}[thm]{Definitions and Remarks}
    \newtheorem{defi-Nots}[thm]{Definition and Notations}
    \newtheorem{defi-Not}[thm]{Definition and Notation}
    \theoremstyle{Lemma}
    \newtheorem{lem}[thm]{Lemma}
    \theoremstyle{remark}
    \newtheorem{rem}[thm]{Remark}
    \theoremstyle{Corollary}
    \newtheorem{cor}[thm]{Corollary}
    \newtheorem{exam}[thm]{Example}
    \newtheorem{counterexam}[thm]{Counterexample}
    \newtheorem{prop}[thm]{Proposition}
    \newtheorem{ques}[thm]{Question}
    \newtheorem{discuss}[thm]{Discussion}
    \newtheorem{defi-rem-nota}[thm]{Definitions, Notations and Remarks}
    \newtheorem{rmi}[thm]{Reminder}
    \normalsize  
    
    \section{Introduction}
    
    Throughout this article $(R,\mathfrak{m})$ is a commutative Noetherian local ring of dimension $d$ with identity where $\mathfrak{m}$  denotes the unique maximal ideal of $R$. Furthermore, $M$ always stands for a finitely generated $d'$-dimensional $R$-module.
    
       Following \cite{Aoyama}, we say that $R$ is a quasi-Gorenstein ring precisely when $\Ht^d_\mam(R)\cong \Et(R/\mam)$ or, equivalently,  $R$ has a canonical module which is a rank one free module. In  the geometric vein, a normal projective variety $X$ is quasi-Gorenstein if the canonical divisor  $K_X$ of $X$ is Cartier. Indeed, a Cohen-Macaulay quasi-Gorenstein ring is Gorenstein. According to  \cite[]{Hassanzadeh} (c.f. \cite{Schenzel}), roughly speaking, quasi-Gorenstein rings arise  from the theory of linkage. That is, loosely speaking,  they are residue rings of a Gorenstein ring modulo an ideal which is linked to an unmixed almost complete intersection. Hence there are so many of them.  From another perspective, if $R$  has a canonical module  and $\Rc$ satisfies the $S_2$ condition  then the trivial extension of $R$ by its canonical module  is quasi-Gorenstein (see \cite{Aoyama}). This, again, shows the ubiquity of quasi-Gorenstein rings. When $R$ is a complete normal domain with a canonical ideal $\omega_R$, the first author of the present paper, in \cite[Remark 3.2.]{TavanfarReduction } and \cite[Theorem 3.5.(i)]{TavanfarReduction},  endowed $R\dsum \omega_R$ with another $R$-algebra structure, by which it is a quasi-Gorenstein domain  (Note that the trivial extension is never domain).  It is also worthwhile to mention that by, \cite[Lemma (2.4)]{Foxby}, the class of quasi-Gorenstein rings contains the class of unique factorization domains with canonical module. Such  non-Cohen-Macaulay unique factorization domains have their origin in the invariant theory, see \cite{BertinAnneaux}. Rees Algebras provides us  with another important source of quasi-Gorenstein rings (see, e.g., \cite{JohnsonUrlich} and \cite{HeiznerUrlich}).  In the light of \cite[Theorem 6.1.]{SinghCyclic},   there are (even) isolated rational singularity  $Q$-Gorenstein rings  with non-Cohen-Macaulay quasi-Gorenstein  cyclic cover. In \cite[Theorem, page 336]{IshiiQuasi-Gorenstein} the author gives a nice description of non-Cohen-Macaulay quasi-Gorenstein Fano $3$-folds\footnote{This means a $3$-dimensional normal projective variety $X$ such that its anti-canonical divisor, $-K_X$, is an ample Cartier divisor.} (over $\mathbb{C}$) with isolated non-rational loci, as the projective cone defined by an ample invertible sheaf $\mathcal{L}$ on an Abelian surface. In \cite{Hermann} the authors  give examples of  quasi-Gorenstein  Buchsbaum affine semigroup rings of any admissible depth. It is noteworthy to point out that affine semigroup rings  of \cite{Hermann} are not the only place where the non-Cohen-Macaulay quasi-Gorenstein Buchsbaum rings come from. Another source of such rings is the Segre product of two hypersurfaces of $a$-invariant zero. For example, over an algebraically closed field $k$, the Segre product ring, $$R:=\big(k[a,b,c]/(a^3+b^3+c^3)\big)\#\big(k[x,y,z,w]/(x^4+y^4+z^4+w^4)\big),$$ is a quasi-Gorenstein Buchsbaum normal domain of dimension $4$ and depth $2$ (see, \cite[Theorem (4.1.5)]{GotoSegre}, \cite[Proposition (4.2.2)]{GotoSegre}, \cite[Theorem (4.3.1)]{GotoSegre} and \cite[Theorem A]{Miazaki}).   We end this part of the introduction by stating that, in the light of Kunz's \cite[Proposition 1.1]{KunzAlmost}, the class of almost complete intersections and the class of quasi-Gorenstein rings are disjoint.

    A vast amount of research has been devoted to  studying the class of  Gorenstein rings. Compared with Gorenstein rings, our understanding of the behaviour and properties of quasi-Gorenstein rings is  limited. In this paper we aim to increase our knowledge about  quasi-Gorenstein rings. For example, in Section 3 we deal with  some natural questions concerning the interaction between quasi-Gorensteinness and regular elements. Namely, we prove the following fact:
     
      \begin{thm}\label{IntroRxR} (See Proposition \ref{quotient_zero_ht} and Corollary \ref{S_2Characterization}(i)) If $R$ is quasi-Gorenstein  and $x$ is a regular element of $R$, then $R/xR$ is   quasi-Gorenstein  if and only if $R/xR$ satisfies the $S_2$ condition and this is equivalent to the assertion that $x\notin\unil_{\map\in\At_R\big(\Ht^{d-1}_\mam(R)\big)}\map$.\end{thm}
      
      Accordingly, the main obstruction here is the failure of $S_2$-condition. Hence, the next natural question is that whether the $S_2$-ification of $R/xR$ is, always, quasi-Gorenstein provided $R$ is quasi-Gorenstein and $R/xR$ has a canonical module? We settle  this question affirmatively.
      
      \begin{thm} \label{IntroS(RxR)} (See Corollary \ref{S_2Characterization}(ii)) If $R$ is quasi-Gorenstein and $x$ is a regular element of $R$ such that $R/xR$ has a canonical module, then the $S_2$-ification of $R/xR$ is quasi-Gorenstein.\end{thm} 
      
      Perhaps it is worth to stress the following application of Theorem \ref{IntroRxR} to the homological conjectures. In \cite[Remark 3.6.(i)]{TavanfarReduction}, applying both  of two characterizations of quasi-Gorensteinness of $R/xR$, given in Theorem \ref{IntroRxR}, as well as the first author's reduction of some homological conjectures to normal complete quasi-Gorenstein domains in \cite{TavanfarReduction}, together   with Ochiai and Shimomoto's   \cite{OchiaiShimomotoBertini}, it is  deduced that the validity of the Hochster's Canonical Element Conjecture in a dense open subset (in some sense) settles the Canonical Element Conjecture in general.
      
      Another immediate question is the deformation property of quasi-Gorenstein rings. In \cite[Proposition 3.4.]{TavanfarReduction}, the first author of the present paper, shows that the Ulrich's deformation of certain Gorenstein rings to unique factorization domains, developed in \cite{UlrichGorenstein}, has a quasi-Gorenstein counterpart. Although, at the time of writing the paper, we do not know whether the quasi-Gorensteinness deforms in general, but  we show that the following variant of deformation of   quasi-Gorensteinness, which we call it as analytic deformation\footnote{As far as we know, our paper is the first place where such a variant of deformation is studied. This is remarkable, because the ordinary deformation problem of quasi-Gorensteinness is mysterious.}, holds true.
      
      \begin{thm} \label{IntroDeform}(See Theorem \ref{QuasiGorensteinDeformation}) Suppose that $x$ is a regular element of $R$ such that $R/x^nR$ is quasi-Gorenstein for each $n\in \mn$. Then $R$ is quasi-Gorenstein.\end{thm}
      
      In order to deduce Theorem \ref{IntroRxR} and Theorem \ref{IntroS(RxR)}, we used a special case of the following fact which is the main result of  Section 2.

      \begin{thm} \label{IntroAtt}(See Theorem \ref{AttPrimeSnLocalCohomology}(ii)) Let $M$ be a formally equidimensional module satisfying the Serre-condition $S_2$. Assume, additionally, that $M$ has a canonical module $\omega_M$\footnote{See, Definition and Remark \ref{GeneralCanonicalModuleDefinition} for the definition of the canonical module of an $R$-module.}. Set, $$n:=\inf\{i: i\ge 1\text{\ and\ } \Ht^{d-i}_\mam(M)\neq 0\}.$$ If the formal fibers of $R$ satisfy the $S_{n+2}$-condition\footnote{e.g., if $R$ is a homomorphic image of a Gorenstein ring.}, then
      $$\At_R\big(\Ht^{d-n}_\mam(M)\big)=\{\map\in\ts(R):\dep_{R_\map}\big((\omega_M)_\map\big)=n+1\text{\ and\ }\hit_{\omega_M}\map\ge n+2\}.$$  \end{thm}

       The section 2. of the paper is, also, devoted to reminding some definitions and well-known facts as well as giving some remarks and lemmas which are required throughout the paper.  
            
     By, e.g. \cite{Foxby2},    a Cohen-Macaulay local ring  is Gorenstein if and only if $\Ht^d_\mam(R)$  has finite injective dimension. In \cite{Enochs}, the authors introduce the notion of Gorenstein injective modules as a generalization of the concept of injective modules. Gorenstein injective modules and, its related invariant, the Gorenstein injective dimension are studied by many authors. For example, in \cite[Theorem 2.6]{Yoshizawa} the author shows that if $R$ is a Cohen-Macaulay complete local ring and $\Ht^d_\mam(R)$ is Gorenstein injective then $R$ is Gorenstein. Subsequently, in \cite{Zargar}, the authors relax  the complete assumption and prove that even if $\Ht^d_\mam(R)$ has finite Gorenstein injective dimension, then the Cohen-Macaulay ring $R$ is Gorenstein. On the other hand, for the injective dimension case, there exists a quasi-Gorenstein counterpart. That is, $R$ is a quasi-Gorenstein ring if and only if $\text{Id}_R\big(\Ht^d_\mam(R)\big)<\infty$ (\cite[Theorem 3]{Aoyama3}). This encouraged us to investigate the Gorenstein injective  version of the Aoyama's theorem which is, indeed,  a theorem whenever $R$ is Cohen-Macaulay. We proved the following fact which also recovers the Cohen-Macaulay case of \cite{Zargar}.
     
     \begin{thm} \label{IntroExtDepth}(See  Theorem \ref{characterization}) Assume that $\Ht^d_\mam(R)$ has finite Gorenstein injective dimension. If, furthermore, $$\ext^1_{\Rc}(\omega_{\Rc},\omega_{\Rc})=\ldots=\ext^{\dep_{\Rc}(\omega_{\Rc})}_{\Rc}(\omega_{\Rc},\omega_{\Rc})=0,$$ then $R$ is quasi-Gorenstein.\end{thm}
     
     The proof of  Theorem  \ref{IntroExtDepth} shows that the vanishing of $\ext^1_R(\omega,\omega_R)$ and $\ext^2_R(\omega_R,\omega_R)$, in the following sense, is sufficient for our purpose .

     \begin{ques} \label{IntroQues} (See  Question \ref{ExtQuestion}) Assume that $R$ has a canonical module $\omega_R$ and that, $$\text{Gdim}_R(\omega_R)<\infty.$$ Then do we have, $\ext^1_R(\omega_R,\omega_R)=\ext^2_R(\omega_R,\omega_R)=0?$\end{ques}

     In fact an affirmative answer to the above question, which is proposed in Section 4, shows that $R$ is quasi-Gorenstein if and only if $\Ht^d_\mam(R)$ has finite Gorenstein injective dimension.
     
         In \cite{Bass},  the famous theorem of Bass  gives a criterion for Gorensteinness in terms  of the irreducibility of all of the parameter ideals. Indeed, there exists a non-Gorenstein ring with an irreducible parameter ideal. But, by virtue of \cite{Marley}, there is an  invariant $\ell_d(R)$ of $R$ such that $R$ is Gorenstein if and only if  some parameter ideal contained in $\mam^{\ell_d(R)}$ is irreducible. In Section 5, we prove the following   quasi-Gorenstein version of  \cite{Marley} (see  Section 5 for the definition of the limit closure).
         
        \begin{thm}\label{IntroLimit} (See  Theorem \ref{characterization}) $R$ is quasi-Gorenstein if and only if $\Rc$ is unmixed and the limit closure of each parameter ideal (of some parameter ideal contained in $\mam^{\ell_d(R)}$) is irreducible. \end{thm}
        
        It is worth pointing out that    \cite[Proposition 2.3.]{Marley}, \cite{Fouli}, and \cite[3.2 Theorem]{O'Carroll} also show that the  limit closure of parameter ideals of non-Cohen-Macaulay rings can be considered as a  counterpart to the parameter ideals of Cohen-Macaulay rings. We would like to  stress that Theorem \ref{IntroLimit} may have the following computational application. According to Theorem \ref{IntroLimit} for verifying the quasi-Gorensteinness of a ring $R$, using Macaulay2 system, we can find a system of parameters $\mathbf{x}$ of $R$ and then check whether, $\mu\big(\{\mathbf{x}\}^{\lim}_R:\mam\big)=\mu(\{\mathbf{x}\}^{\lim}_R)+1$, wherein $\{\mathbf{x}\}^{\lim}_R$ denotes the limit closure of $\mathbf{x}$. \footnote{Here,  the notation $\mu$ stands for the minimal number of generators. Recall that a primary ideal $\maa$ of $(R,\mam)$ is irreducible if and only if $\mu(\maa:\mam)=\mu(a)+1$.}. This can be, considerably, faster than computing the free resolutions, especially when the defining ideal of $R$ has too many generators or generators with too many summands. It is noteworthy to mention, here, that, e.g., the limit closure of any system of parameters $\mathbf{x}$ of a Buchsbaum local ring $R$ is just the ideal, $(x_1^2,\ldots,x_d^2):x_1\cdots x_d$ (see, \cite[Lemma (4.6).]{GotoOnTheAssociated}).

    \section{Attached Prime ideals of Local Cohomology modules}
    
    \begin{rmi}\label{PreliminariesCanonical}
      \begin{enumerate} Here we aim to review some, probably known, definitions and facts.
         \item[(i)] Throughout the paper, we use the notation $\text{Assht}$ to denote the set, $$\text{Assht}_R(M):=\{\map\in\As_R(M):\dime(R/\map)=d'\}.$$ $M$ is said to be equidimensional (respectively, unmixed) if $\text{Assht}(M)=\text{minAss}_R(M)$ (respectively, $\text{Assht}(M)=\As_R(M)$). $M$ is called formally equidimensional (respectively, formally unmixed) if $\Mc$\footnote{Here the notation $\widehat{M}$ stands for the completion of $M$ with respect to the $\mam$-adic topology.} is an equidimensional (respectively, unmixed) $\Rc$-module.
         \item[(ii)]  For each ideal $\maa$ of $R$ we set, $$\hit_M\maa:=\hit_{R/0:_RM}\big((\maa+0:_RM)/0:_RM\big).$$ Note that, $\dim_{R_\map}M_\map=\hit_M\map$, for each $\map\in \Supp_R(M)$. Thus,
           $$\hit_M\maa=\inf\{\hit_M\map:\map\in \Var(\maa)\ins \Supp_R M\}=\inf\{\dime_{R_\map}M_\map:\map\in \Var(\maa)\ins \Supp_R M\}.$$
           
         \item[(iii)] We say that   $M$ satisfies the Serre-condition $S_n$ precisely when, $$\dep_{R_\map}(M_\map)\ge \min\{n,\dime_{R_\map}(M_\map)\},$$ for each $\map\in \ts(R)$. In fact,  the following statements  are equivalent.
           \begin{enumerate}
             \item[(a)] $M$ satisfies $S_n$.
             \item[(b)] $\grd(\maa,M)\ge \min\{n,\hit_M\maa\}$ for each ideal $\maa$ of $R$.
             \item[(c)] if $\maa$ is an ideal generated by an $M$-regular sequence of length $j\le n-1$ then  for each $\map\in\As_R(M/\maa M)$ we have $\hit_M(\map)=j$.
            
           \end{enumerate}
        \item[(iv)] We use the notation,  $M^\vee$, to denote the Matlis dual  $\homm_R\big(M,\text{E}(R/\mam)\big)$ of an $R$-module $M$. The canonical module for $R$, denoted by $\omega_R$, is a finitely generated $R$-module such that, $\omega_R^\vee\cong \Ht^d_\mam(R)$. The  module $\omega_R$, if exists,  is an $d$-dimensional $R$-module satisfying the Serre-condition $S_2$ and is unique up to isomorphism (\cite[12.1.6 Theorem]{LocalCohomology}, \cite[12.1.9 Proposition(i)]{LocalCohomology} \cite[12.1.18 Theorem(i)]{LocalCohomology}). So $R$ is a quasi-Gorenstein ring if and only if $R$ has a  canonical module which is a rank one free $R$-module. In particular a quasi-Gorenstein ring is $S_2$ and unmixed (\cite[Lemma 1.1]{Aoyama2}). A fortiori,  by  \cite[(1.8)]{Aoyama} we have, 
        \begin{center}
          \begin{equation}  
            \label{AnnihilatorCanonicalModule}
            0:_R\omega_R=\insl_{\maq}\maq,
          \end{equation}
        \end{center}
                   where the intersection runs over the primary components $\maq$ of the zero ideal satisfying   $\dime(R/\maq)=d$.
This, again, implies the unmixedness of quasi-Gorenstein rings. If $R$ is a homomorphic image of a Gorenstein local ring $S$ of dimension $n$ then by virtue of \cite[11.2.6 Local Duality Theorem ]{LocalCohomology} we have, $\omega_R:=\ext^{n-d}_S(R,S)$. 
        \item[(v)] (See  \cite[Lemma 1.9.]{SchenzelLocalCohomology}) Assume that $R$ is a homomorphic image of a Gorenstein local ring $S$ of dimension $n$.  For each $l\in \mathbb{Z}$ set, $\omega_M^l:=\ext^{n-l}_S(M,S)$. By \cite[11.2.6 Local Duality Theorem]{LocalCohomology} there are functorial isomorphisms \begin{center}
  \begin{equation}
    \label{LocalDuality}
    \Ht^l_\mam(N)\cong \ext_S^{n-l}(N,S)^\vee,
  \end{equation}
\end{center}
 for all finitely generated $R$-modules $N$. The $R$-module $\omega_M^{\dim M}$ is said to be the canonical module of $M$ and is denoted by $\omega_M$. The $R$-module $\omega_M$ satisfies $S_2$ and $\As_R(\omega_M)=\text{Assht}_R(M)$. In particular, $\dime(\omega_M)=\dime(M)$ (see \cite[Lemma 1.9.(c)]{SchenzelLocalCohomology}). 
 
\ \ \  \ Let $\map\in \Supp_R(\omega_M)$. Then, by \cite[Lemma 1.9(a)]{SchenzelLocalCohomology}, $(\omega_M)_\map\cong \omega_{M_\map}^{\dime(M)-\dime(R/\map)}$. It follows, from (\ref{LocalDuality}), that $\Ht^{\dime(M)-\dime(R/\map)}_{\map  R_\map}(M_\map)\neq 0.$ Consequently, $\dime(R/\map)+\dime(M_\map)=\dime(M)$. In particular, we get $(\omega_M)_\map=\omega_{M_\map}$.
 
    \ \ \  The argument of the preceding paragraph in conjunction with \cite[Lemma 1.9(b)]{SchenzelLocalCohomology} yields, $$ \text{Supp}_R(\omega_M)=\{\map\in \Supp_R(M):\dime(M_\map)+\dime(R/\map)=\dime(M)\}.$$

      \end{enumerate}
    \end{rmi}
    
    Similarly, as in part (iv) of the preceding reminder, the definition of the canonical module of a finitely generated $R$-module can be extended to all local rings.
    
    \begin{defi-rem} \label{GeneralCanonicalModuleDefinition}
        Let $M$ be a finitely generated $d'$-dimensional $R$-module. Then we say that a finitely generated $R$-module $\omega_M$ is the canonical module for $M$ if $\omega_M^\vee\cong \Ht^{d'}_\mam(M)$. 
        
        A similar argument as in the proof of \cite[12.1.6 Theorem]{LocalCohomology} shows that the canonical module $\omega_M$ of $M$, if exists, is unique up to isomorphism.  Note that, $\omega_{\widehat{M}}\cong \omega_M\tenr \widehat{R}$. By Reminder \ref{PreliminariesCanonical}(v), $\widehat{M}$ always has a canonical module. In  Lemma \ref{CanonicalModuleAndModuleHeghtInequality} we prove some counterparts to the  properties of $\omega_M$ stated in the part (v) of the preceding reminder. Note that,  if $R$ has a canonical module $\omega_R$  then, by  \cite[6.1.10 \# Exercise]{LocalCohomology},   for each $d$-dimensional $R$-module $N$  
            we have  $\homm_R(N,\omega_R)\cong\omega_N$.
    \end{defi-rem}

    \begin{rem}\label{FirstRemark} 
       
        Let $M$ be a (finitely generated) $d'$-dimensional $R$-module. By virtue of the Cohen's structure theorem there exists a Gorenstein ring $S$ such that $\dime(R)=\dime(\widehat{R})=\dime(S)$ and $\widehat{R}$ is a homomorphic image of $S$.
        
        (a) Let $x\in R\backslash \text{Z}_R(M)$.  Then from the exact sequence,
        $$0\rightarrow M\overset{x}{\rightarrow} M\rightarrow M/xM\rightarrow 0,$$
         we get the exact sequence, $H^{d'-1}_{\mathfrak{m}}(M)\overset{x}{\rightarrow} H^{d'-1}_{\mathfrak{m}}(M)\rightarrow H^{d'-1}_{\mathfrak{m}}(M/xM)\overset{\Delta}{\rightarrow}H^{d'}_{\mathfrak{m}}(M)\overset{x}{\rightarrow}H^{d'}_{\mathfrak{m}}(M),$ which yields the exact sequence, 
         \begin{equation}
         \label{FirstExactSequence}
           0\rightarrow \Ht^{d'-1}_\mam(M)/x\Ht^{d'-1}_\mam(M)\rightarrow \Ht^{d'-1}_\mam(M/xM)\rightarrow 0:_{\Ht^{d'}_\mam(M)}x\rightarrow 0.
          \end{equation}  
          Applying the Matlis dual functor  to the exact sequence (\ref{FirstExactSequence})  we obtain the exact sequence,
          \begin{equation}
             \label{SecondExactSequence}
             0\rightarrow \omega_{\widehat{M}}/x\omega_{\widehat{M}}\overset{\iota}{\rightarrow} \omega_{\widehat{M}/x\widehat{M}}\rightarrow C\rightarrow 0,
          \end{equation}
          where, $C=\big(\Ht^{d'-1}_\mam(M)/x\Ht^{d'-1}_\mam(M)\big)^\vee$. By \cite[11.2.6 Local Duality Theorem]{LocalCohomology}, 
          \begin{align*}
          C\cong \big(\Ht^{d'-1}_{\mam\widehat{R}}(\widehat{M})/x\Ht^{d-1}_{\mam\widehat{R}}(\widehat{M})\big)^\vee &= \homm_{\widehat{R}}\big(\Ht^{d'-1}_{\mam\widehat{R}}(\widehat{M})\ten_{\widehat{R}}\widehat{R}/x\widehat{R},\text{E}(\widehat{R}/\mam\widehat{R})\big)&\\ &\cong \homm_{\widehat{R}}\Big(\widehat{R}/x\widehat{R},\big(\Ht^{d'-1}_{\mam\widehat{R}}(\widehat{M})\big)^\vee\Big)&\\ &\cong 0:_{\ext^{d-d'+1}_S(\widehat{M},S)}x.
          \end{align*}

          (b) Now assume, furthermore, that  $R$ is formally equidimensional. Set, $M=R$. Then \cite[page 250, Lemma 2]{Matsumura} implies that, $\hit(\map)+\dime(\Rc/\map)=\dime(\Rc)$, for each $\map\in \ts(\Rc)$. Set, $\Rc=S/\maa$. Let, $\map=\mathfrak{P}/\maa\in \ts(\Rc)$ where $\mathfrak{P}\in \ts(S)$. We have, $$\dime(S_\mathfrak{P})=\hit_{S}\mathfrak{P}=\dime(S)-\dime(S/\mathfrak{P})=\dime(\Rc)-\dime(\Rc/\map)=
          \hit_{\Rc}(\map)
          =\dime(\Rc_\map).$$ Therefore, $\ext^1_S(\Rc,S)_\map\cong \ext^1_S(\Rc,S)_\mathfrak{P}\cong \ext^1_{S_\mathfrak{P}}(\Rc_\map,S_\mathfrak{P})\cong\homm_{\Rc_\map}\big(\Ht^{\dime(\Rc_\map)-1}_{\map \Rc_\map}(\Rc_\map),\text{E}(\Rc_\map/\map \Rc_\map)\big)$. Consequently, 
            \begin{align*} 
              C_\map\cong \big(0:_{\ext^1_S(\Rc,S)}x\big)_\map&\cong \homm_{\Rc_\map}\Big(\Rc_\map/x\Rc_\map,\homm_{\Rc_\map}\big(\Ht^{\dime(\Rc_\map)-1}_{\map \Rc_\map}(\Rc_\map),\text{E}(\Rc_\map/\map \Rc_\map)\big)\Big)&\\ &\cong\homm_{\Rc_\map}\big(\Ht^{\dime(\Rc_\map)-1}_{\map \Rc_\map}(\Rc_\map)/x\Ht^{\dime(\Rc_\map)-1}_{\map \Rc_\map}(\Rc_\map),\text{E}(\Rc_\map/\map \Rc_\map)\big).
            \end{align*} 
           We summarize the above observation as follows.
           
           (c) If in addition $R$ is presumed to be formally equidimensional then, setting $M=R$, the module $C$ in the exact sequence (\ref{SecondExactSequence}) is locally dual to, $\Ht^{\dime(\Rc_\map)-1}_{\map \Rc_\map}(\Rc_\map)/x\Ht^{\dime(\Rc_\map)-1}_{\map \Rc_\map}(\Rc_\map)$, for each $\map\in \ts(\widehat{R})$.
    \end{rem}
    
    In the following lemma we show that some of the properties of the canonical module, which are well-known in the case where $R$ is a homomorphic image of a Gorenstein ring, hold in general. 
         
    \begin{lem} Let $M$ has a canonical module. \label{CanonicalModuleAndModuleHeghtInequality}
      The following statements hold.
      \begin{enumerate}
         \item[(i)] 
           $\As_R(\omega_M)=\text{Assht}_R(M)$. In particular, $\dim(\omega_M)=\dime(M)$ and $\Supp_R(\omega_M)\subseteq \Supp_R(M)$.
           \item[(ii)] $\Supp_{R}(\omega_M)=\{\map\in\Supp_{R}(M):\hit_M\map+\dime(R/\map)=d'\}$.
        \item[(iii)]
          For each ideal $\maa$ of $R$ we have $\hit_{\omega_M}\maa\ge \hit_M\maa$. But, $\hit_{\omega_M}\map=\hit_M\map$ for each $\map\in \Supp(\omega_M)$.
        \item[(iv)] $\omega_M$ satisfies the Serre condition $S_2$.
      \end{enumerate}  
       
      \begin{proof}  
      
        (i) By, \cite[Theorem 23.2]{Matsumura}, we can deduce that $\As_RN=\{\maq\ins R;\maq\in \As_{\widehat{R}}(N\tenr \widehat{R})\}$ for each finitely generated $R$-module $N$. Thus, \cite[Lemma 1.9.(c)]{SchenzelLocalCohomology} implies that,
          \begin{align*}
           \As_R(\omega_M)&=\{\maq\ins R:\maq\in \As_{\widehat{R}}(\omega_M\tenr \Rc)\} &\\ &=\{\maq\ins R:\maq\in \As_{\widehat{R}}(\omega_{\Mc})\}&\\ &=\{\maq\ins R:\maq\in \text{Assht}_{\Rc}(\Mc\}.
          \end{align*}
          Thus it is enough to prove that, $\{\maq\ins R:\maq\in \text{Assht}_{\Rc}(\Mc)\}=\text{Assht}_R(M)$. Let, $\map\in \text{Assht}_R(M)$. There exists $\maq\in \ts(\Rc)$ such that $\map\Rc\subseteq \maq$ and $\dime(\Rc/\maq)=\dime(\Rc/\map\Rc)$. So we have $\dime(\widehat{M})=\dime(M)=\dime(R/\map)=\dime(\Rc/\map\Rc)=\dime(\Rc/\maq)$. Since, $\maq\in \min(\map\Rc)$, so \cite[Theorem 9.5.]{Matsumura} implies that $\maq\ins R=\map$. Thus, $\maq\in \Supp_{\Rc}(\Mc)$. Consequently $\maq\in \text{Assht}_{\Rc}(\Mc)$, as required.
         
          Conversely, let $\maq\in \text{Assht}_{\Rc}(\Mc)$ and $\map=\maq\ins R$. We have $\map\in \As_R(M)$ because $\maq\in \As_{\Rc}(\Mc)$. Also, $\dime_R(M)\ge \dime(R/\map)=\dime(\Rc/\map\Rc)\ge \dime(\Rc/\maq)=\dime_{\Rc}(\Mc)=\dime_R(M)$.

         (ii) Let $N$ be a finitely generated $R$-module and $\map\in \Supp_R(N)$. Then we can choose $\maq\in \min(\map\Rc)$ such that $\dime(\Rc/\maq)=\dime(\Rc/\map\Rc)=\dime(R/\map)$. Therefore, $\maq\ins R=\map$ and so $\Rc_\maq$ is a faithfully flat extension of $R_\map$. Since, $(N\tenr\Rc)\cong N_\map\ten_{R_\map}\Rc_\maq$ so we have $\maq\in \Supp_{\Rc}(N\tenr \Rc)$. Due to the formula \cite[Theorem A.11]{Herzog} we have,
         \begin{align*}
           \dime\big((N\tenr \Rc)_\maq\big)=\dime(N_\map\ten_{R_\map} \Rc_\maq)= \dime(N_\map)+\dime(\Rc_\maq/\map \Rc_\maq)=\dime(N_\map).
         \end{align*}
         Now, our claim follows from the above observation in conjunction with the fact that $\omega_{\Mc}=\omega_M\tenr \Rc$.
      
         (iii) Assume that $R$ is complete. Then using Reminder \ref{PreliminariesCanonical}(v) we have $(\omega_M)_\map\cong \omega_{M_\map}$ for each $\map\in\Supp_R(\omega_M)$. It turns out that, $\hit_{\omega_M}\map=\dime_{R_\map}\big((\omega_M)_\map\big)=\dime_{R_\map}(M_\map)=\hit_{M}\map$, for each $\map\in \Supp_R(\omega_M)$. It follows that, 
           \begin{align*}
             \hit_{\omega_M}\maa &=\text{inf}\{\hit_{\omega_M}\map:\map\in \text{Var}(\maa)\bigcap \Supp_R(\omega_M)\}&\\ &=\text{inf}\{\hit_{M}\map:\map\in \text{Var}(\maa)\bigcap \Supp_R(\omega_M)\} &\\ &\ge \text{inf}\{\hit_{M}\map:\map\in \text{Var}(\maa)\bigcap \Supp_R(M)\}&\\&=\hit_M\maa.
           \end{align*}
          But, $$\hit_{\omega_M}\maa=\hit_{R/0:_{R}\omega_M}\big((\maa+0:_{R}\omega_{M})/0:_{R}\omega_{M}
          \big)=
          \hit_{\widehat{R}/0:_{\widehat{R}}\omega_{\widehat{M}}}
          \big((\maa\widehat{R}+0:_{\widehat{R}}\omega_{\widehat{M}})/0:_{\widehat{R}}
          \omega_{\widehat{M}}\big)=\hit_{\omega_{\widehat{M}}}\maa\widehat{R}.$$
          By the same token, $\hit_M\maa=\hit_{\widehat{M}}\maa\widehat{R}$. Hence the statement follows from the preceding inequality. Clearly, since $\Supp_R(\omega_M)\subseteq \Supp_R(M)$, so in fact the inequality,  $\hit_{\omega_M}\map\ge \hit_M\map$, is an equality for each $\map\in \Supp_R(\omega_M)$.
          
          (iv) The statement follows from the fact that $\omega_{\Mc}\cong \omega_M\tenr \Rc$ satisfies $S_2$.

      \end{proof}
    \end{lem}

     The second part of the following lemma, which is a special case of the first part, will be used several times throughout the remainder of the paper.  It is worth pointing out that one implication of the second part of the following lemma can be deduced by \cite[Lemma 2.1]{Aoyama2}, in the case where $d'=d$.

     \begin{lem} \label{AttachPrimeDepthCanonicalModuleLemma}
        The following statements hold.
         \begin{enumerate}
             \item[(i)] Let $\map\in \Supp_R(M)$ such that $\hit_M \map \ge 3$. Suppose, furthermore, that $\map\nsubseteq \text{Z}_R(M)$. If  $\map\in \At\big(\Ht^{d'-1}_\mam(M)\big)$,  then  $\grd\big(\map\widehat{R},\omega_{\widehat{M}}\big)=2$.
             \item[(ii)] (c.f. Remark \ref{Jafari})  Suppose that $\dep_R(M)\ge 1$ and $\dime(M)\ge 3$. Then, $\mam\in \At\big(\Ht^{d'-1}_\mam(M)\big)$ if and only if $\dep_{\widehat{R}}(\omega_{\widehat{M}})=2$.
          \end{enumerate}
       \begin{proof}
          We first deal with the case where $R$ is complete.   Let $x\in \map\backslash \text{Z}_R(M)$.  Then, by Reminder \ref{PreliminariesCanonical}(v), $x\in \map\backslash \text{Z}_{R}(\omega_{M})$. The fact that $\map\in \At_R\big(\Ht^{d'-1}_\mam(M)\big)$ in conjunction with  \cite[7.2.5. Exercise]{LocalCohomology} implies that, $\map=0:_R\big(\Ht^{d'-1}_\mam(M)/A\big)$ for some submodule $A$ of $\Ht^{d'-1}_\mam(M)$. Hence we get an exact sequence, $$\Ht^{d'-1}_\mam(M)/x\Ht^{d'-1}_\mam(M)\rightarrow \Ht^{d'-1}_\mam(M)/A\rightarrow 0,$$ because $x$ belongs to $\map$. Thus, $\map/xR\in \At_{R/xR}\big(\Ht^{d'-1}_\mam(M)/x\Ht^{d'-1}_\mam(M)\big).$ Consequently, by virtue of \cite[10.2.20 Corollary]{LocalCohomology}, we can conclude that, $p/xR\in \As_{R/xR}\Big(\big(\Ht^{d'-1}_\mam(M)/x\Ht^{d'-1}_\mam(M)\big)^\vee\Big)$. It turns out that, 
          \begin{equation}
            \label{ZeroGradeSoNonZeroHom}
            \homm_{R/xR}\Big(R/\map,\big(\Ht^{d'-1}_\mam(M)/x\Ht^{d'-1}_\mam(M)\big)^\vee\Big)\neq 0.
          \end{equation}  
          
          We have,  $\dime_{R_\map/xR_\map}(M_\map/xM_\map)=\dime_{R_\map}(M_\map)-1\ge 2$. Therefore Lemma \ref{CanonicalModuleAndModuleHeghtInequality}(iii) yields, $$\hit_{\omega_{M/xM}}(\map/xR)\ge \hit_{M/xM}(\map/x R)\ge 2.$$ In particular, since $\omega_{M/xM}$ is $S_2$ as $R/xR$-module so using Reminder \ref{PreliminariesCanonical}(iii) we get, $$\grd_{R/xR}\big(\map/xR,\omega_{M/xM}\big)\ge \min\{2,\hit_{\omega_{M/xM}}(\map/xR)\}\ge 2.$$ It turns out that,
            \begin{equation} 
               \label{MinimumTwoGradeZeroExt}
               \homm_{R/xR}(R/\map,\omega_{M/xM})=\ext^1_{R/xR}(R/\map,\omega_{M/xM})=0
            \end{equation}
            Now, using (\ref{ZeroGradeSoNonZeroHom}) and (\ref{MinimumTwoGradeZeroExt}), in light of the exact sequence (\ref{SecondExactSequence}) of Remark \ref{FirstRemark}, we can deduce that, $$\ext^1_{R/xR}(R/\map,\omega_{M}/x\omega_{M})\neq 0.$$
            Consequently, $\grd_{R/xR}(\map/xR,\omega_{M}/x\omega_{M})\le 1$, i.e. $\grd_{R}(\map,\omega_{M})\le 2$. On the other hand, $\hit_{\omega_{M}}\map\ge  \hit_{M}\map\ge 3$ which implies that, $\grd_{R}(\map,\omega_{M})\ge \min\{2,\hit_{\omega_M}\map\}\ge 2$. So the statement holds in this case.
            
              In the general case, in light of \cite[11.2.10 Exercise]{LocalCohomology}, there exists $\maq\in \At_{\Rc}\big(\Ht^{d'-1}_{\mam\Rc}(\Mc)\big)$ with $\map=\maq\ins R$. We have, $\maq\in \Supp_{\Rc}(M\tenr \Rc)$, $\maq\nsubseteq \text{Z}_{\Rc}(M\tenr \Rc)$ and $3\le \hit_M\map\le \hit_{\Mc}\maq$. Then the above argument shows that $\grd_{\Rc}(\maq,\omega_{\Mc})=2$ and thence $\grd_{\Rc}(\map\widehat{R},\omega_{\Mc})\le 2$. This  completes the proof because by Lemma \ref{CanonicalModuleAndModuleHeghtInequality}(iii) and (iv) we have, $\hit_{\omega_{\Mc}}(\map\widehat{R})\ge \hit_{\Mc}(\map\widehat{R})=\hit_M(\map)\ge 3$ and $\omega_{\Mc}$ is $S_2$.
            
            (ii) We may and we do assume that $R$ is complete. One implication follows from the preceding part. Note that, $$\mam/xR\in \As_{R/xR}\Big(\big(\Ht^{d'-1}_\mam(M)/x\Ht^{d'-1}_\mam(M)\big)^\vee\Big),$$ if and only if, $\homm_{R/xR}\Big(R/\mam,\big(\Ht^{d'-1}_\mam(M)/x\Ht^{d'-1}_\mam(M)\big)^\vee\Big)\neq 0$. Hence we can conclude the reverse implication, also, by tracing back the proof of part (i).

       \end{proof}
     \end{lem}

     \begin{rem} 
        It is, perhaps, worth pointing out that the reverse of the first part of the preceding lemma does not hold in general. For example, assume that $R$ is a  quasi-Gorenstein Buchsbaum non-Cohen-Macaulay ring of dimension 4 and depth 2 (see  \cite[Lemma 2.2.]{Hermann} or the Segre product ring given in the introduction). Let $\map\in \ts(R)$ with $\hit(\map)=3$. $R$ is $S_2$ and $\hit_R(\map)\ge 3$, hence by  Reminder \ref{PreliminariesCanonical}(iii) we have $2\le\grd_R(\map,R)\le \dep(R)=2$. So, $\grd_{\Rc}(\map\Rc,\Rc)=\grd_{R}(\map,R)=2$. But  $\map\notin \At_R\big(\Ht^{d-1}_\mam(R)\big)$, because $\Ht^{d-1}_\mam(R)$ has finite length.  Also, it seems that we can not deduce Lemma \ref{AttachPrimeDepthCanonicalModuleLemma}(i)  from Lemma \ref{AttachPrimeDepthCanonicalModuleLemma}(ii) by localization, because we shall need the further assumption, $\hit_M\map+\dim (R/\map)=\dim (M),$ which is not assumed.

     \end{rem}
     \begin{rem} \label{Jafari}This remark is due to  Raheleh Jafari. In fact, what follows inspired us to prove the preceding lemma, as a slight generalization.  Set, $M^0:=\dsuml_{\substack{\map\in \Supp_R(M),\\ \hit_M\map=0}}M_\map$. Define the map $\partial^{-1}_M:M\rightarrow M^0$ such that its projection onto $M_\map$ is the natural localization map for each $\map\in \Supp(M)$ with $\hit_M\map=0$. Assume, inductively, that $M^i$ and $\partial^{i-1}_M:M^{i-1}\rightarrow M^i$ are constructed. Then set, $M^{i+1}:=\dsuml_{\substack{\map\in \Supp_R(M),\\ \hit_M\map=i+1}}\big(\text{Coker}(\partial^{i-1}_M)\big)_\map,$ and $\partial^{i}_M:M^i\rightarrow M^{i+1}$ to be the following map. For $m\in M^i$ and $\map\in \Supp_R(M)$ with $\hit_M(\map)=i+1$, the component of $\partial^i_M(m)$ in $\big(\text{Coker}(\partial^{i-1}_M)\big)_\map$ is $\overline{m}/1$, where, $-:M^i\rightarrow \text{Coker}(\partial^{i-1}_M)$ is the natural map. Then, we get the complex,
       $$\mathcal{C}(M):=0\rightarrow M\overset{\partial^{-1}_M}{\rightarrow}M^0\overset{\partial^{0}_M}{\rightarrow}M^1\ldots\overset{\partial^{i-2}_M}{\rightarrow} M^{i-1} \overset{\partial^{i-1}_M}{\rightarrow}M^i\overset{\partial^{i+1}_M}{\rightarrow}\ldots,$$
       which is called the Cousin complex of $M$ with respect to the hight filtration.  The  Cousin complex is introduced and investigated firstly in \cite{Cousin} as the Commutative Algebraic analogue of the Cousin complex introduced by A. Grothendieck  and R. Hartshorne \cite[Chapter IV]{Hartshorne}.  Denote by $\mathcal{H}^i_M$ the $i$-th cohomology of the Cousin complex $\mathcal{C}(M)$.
       
        Now assume, additionally, that $M$ is $S_2$ and formally equidimensional and the formal fibers of $R$ are Cohen-Macaulay.  Then by virtue of   \cite[Theorem 2.4]{Dibaei} we have $\mathcal{H}^{d'-2}_M\neq 0$  if and only if $ \Ht^{2}_{\mam\Rc}(\omega_{\Mc})\neq 0$. Moreover  in spite of  \cite[Lemma 2.2.(iii)]{Jafari} and \cite[Theorem 2.1.]{Dibaei}   we have $\mam\in\At_R\big(\Ht^{d'-1}_R(M)\big)$ if and only if $\mathcal{H}^{d'-2}_M\neq 0$. This, again, proves   Lemma \ref{AttachPrimeDepthCanonicalModuleLemma}(ii) in the case where $M$ is $S_2$ and formally equidimensional and formal fibers of $R$ are Cohen-Macaulay.
     \end{rem}
     
     The following lemma is used in the proof of the second part of Theorem \ref{AttPrimeSnLocalCohomology}.

    \begin{lem} \label{Aoyama's result} 
      Assume that $M$ is not Cohen-Macaulay  and that $M$ has a canonical module $\omega_M$. Put, $$s:=\max\{i:i<d' \text{\ and\ }\Ht^i_\mam(M)\neq 0\}.$$ If $\dep_{\Rc}\big(\Ht^s_\mam(M)^\vee\big)=0$, then,
       \begin{center}
          $\dep(\omega_M)=\begin{cases}
            d', & s=0\\
            d'-s+1, & s>0
          \end{cases}.$
       \end{center}
       \begin{proof}
          It suffices to prove the claim in the complete case, so we assume that $R$ is complete. Set, $R':=R/0:_RM$. Recall that since, $0:_RM\subseteq 0:_R\omega_M$, so both of $M$ and $\omega_M$ are $R'$-modules. Furthermore, $\omega_M$ is also the canonical module of $M$ over $R'$. In addition,  the numbers $s$, $\dep(\omega_M)$ and $\dep\big(\Ht^s_\mam(M)^\vee\big)$ are independent of taking the base ring to be either $R$ or $R'$. Thus, without loss of generality, we can presume additionally that $\dime(R)=\dime(M)$ and that $R$ also has a canonical module (since $R$ is assumed to be complete). So the statement follows from  \cite[Lemma 2.1]{Aoyama2}.
       \end{proof}
    \end{lem} 
    
     In the  following theorem we give an explicit description of the attached prime ideals of some of  local cohomology modules,  namely   the first non-zero local cohomology supported at $\mam$ starting from the point $\dime(M)-1$.  We also, as an application,  reprove  \cite[Theorem 1.15]{SchenzelLocalCohomology} (In the case where $R$ is not necessarily  homomorphic image of a Gorenstein ring and our proof is quite different).  Clearly we may replace $\omega_R$ with $R$ in the following theorem provided $R$ is quasi-Gorenstein. Furthermore, the first part of the next theorem shall be used in the proof of the Corollary \ref{S_2Characterization}(i).
     
     \begin{thm}\label{AttPrimeSnLocalCohomology} 
       Assume that $M$ is formally equidimensional and $S_2$ and that $M$ has a canonical module $\omega_M$.  Then the following statements hold.
         \begin{enumerate}
           \item[(i)] Suppose that the formal fibers of $R$ are $S_3$\footnote{Note that all formal fibers of a homomorphic image of a Cohen-Macaulay ring are Cohen-Macaulay}. Then we have, $$\At_R\big(\Ht^{d'-1}_\mam(M)\big)=\{\map\in \ts(R):\dep_{R_\map}\big((\omega_M)_\map\big)=2,\ \text{and,\ }\hit_{\omega_M}\map\ge 3\}.$$
           \item[(ii)] Assume that, $n\ge 2$ and that formal fibers of $R$ satisfy the $S_{n+1}$ condition. Then, $\omega_M$ satisfies the $S_n$-condition,  if and only if,  $\Ht^i_\mam(M)=0$, for each $d'-n+2\le i\le d'-1$. In this case, 
           $$\At\big(\Ht^{d'-n+1}_\mam(M)\big)=\{\map\in\ts(R):\dep_{R_\map}\big((\omega_M)_\map\big)=n\text{\ and\ }, \hit_{\omega_M}\map\ge n+1\}.$$
           
         \end{enumerate}
           \begin{proof}
            Note that, Since  $M$ and formal fibers of $R$ satisfies the $S_2$-condition, so $\widehat{M}$ is $S_2$ too. Since $\Mc$ is also $S_1$ so, $\min_{\Rc}(\Mc)=\As_{\Rc}(\Mc)$. Consequently, in view of the our hypothesis, $\Mc$ is unmixed. Let $\maq\in \Supp_{\Rc}(\Mc)$. Then there exists $\map\in \text{Assht}_{\Rc}(\Mc)$ such that $\map\subseteq \maq$. By Lemma \ref{CanonicalModuleAndModuleHeghtInequality}(i), we have $\map\in \As_{\Rc}(\omega_{\Mc})$. Hence, $\maq\in \Supp_{\Rc}(\omega_{\Mc})$. So the identity, 
                  \begin{center}
                    \begin{equation}
                    \label{htCodimeIneqality}
                    \hit_{\Mc}\maq+\dime(\Rc/\maq)=d',
                    \end{equation}
                  \end{center}                    
                     holds due to  Lemma \ref{CanonicalModuleAndModuleHeghtInequality}(ii).
                     
              (i) If $d'\le 2$ then $M$ is a Cohen-Macaulay $R$-module and the statement follows from Lemma \ref{CanonicalModuleAndModuleHeghtInequality}(iv). So we suppose that $d'\ge 3$. Let $\map\in \ts(R)$ such that $\dep\big((\omega_M)_\map\big)=2$ and  $\hit_{\omega_M}\map\ge 3$. 
               Then $\map\in \Supp_R(\omega_M)$ and in view of Lemma \ref{CanonicalModuleAndModuleHeghtInequality}(iii), we have $\hit_M\map=\hit_{\omega_M}\map\ge 3$.   Let, $\maq\in \min_{\widehat{R}}(\map\widehat{R})$. Then, $\maq\ins R=\map$. Using \cite[Proposition 1.2.16.(a)]{Herzog} and the fact that $\Rc_\maq$ is a faithfully flat extension of $R_\map$ we can deduce that $\dime_{R_\map}(M_\map)=\dime_{\Rc_\maq}(\Mc_\maq)$, $\dep_{R_\map}(M_\map)=\dep_{\Rc_\maq}(\Mc_\maq)$ and $\dep_{\Rc_\maq}\big((\omega_{\Mc})_\maq\big)=\dep_{R_\map}\big((\omega_M)_\map\big)$.  On the other hand, in light of   Reminder \ref{PreliminariesCanonical}(v) we get, $(\omega_{\Mc})_\maq\cong \omega_{\Mc_\maq}$ and thence by Lemma \ref{AttachPrimeDepthCanonicalModuleLemma}(ii),
                $$ \maq\widehat{R}_\maq\in\At_{\widehat{R}_\maq}\big(\Ht^{\dim(\widehat{M}_\maq)-1}_{\maq\widehat{R}_\maq}(\widehat{M}_\maq)\big).$$ 
                Thus \cite[11.2.10 Exercise (iii)]{LocalCohomology}    yields  $\map R_\map\in \At_{R_\map}\big(\Ht^{\hit_M\map-1}_{\map R_\map}(M_\map)\big).$ So the statement follows from  Lemma \ref{CanonicalModuleAndModuleHeghtInequality}(ii) and \cite[11.2.11 Exercise]{LocalCohomology}.

                  Now, conversely, assume that $\map\in \At_R\big(\Ht^{d'-1}_\mam(M)\big)$.
                  By \cite[11.2.10 Exercise (iii)]{Local Cohomology} there exists $\maq\in \widehat{R}$ such that $\maq$ lies over $\map$ and $\maq\in \At_{\widehat{R}}\big(\Ht^{d'-1}_{\mam\widehat{R}}(\widehat{M})\big)$. Hence, the identity (\ref{htCodimeIneqality}) in conjunction with \cite[11.2.8 Exercise(i)]{LocalCohomology} implies that $\maq\widehat{R}_\maq\in \At_{\widehat{R}}\big(\Ht^{\dime(\widehat{M}_\maq)-1}_{\maq\widehat{R}_\maq}(\widehat{M}_\maq)\big)$. This,   and equality \ref{htCodimeIneqality},  yields $\maq\widehat{R}_\maq\notin \As_{\widehat{R}_\maq}(\widehat{M}_\maq)$, because $\Mc$ is unmixed.
                   Next, since $\widehat{M}$ is $S_2$ and $\Mc_\maq$ is not Cohen-Macaulay, so we deduce that $\dim_{\widehat{R}_\maq}(\widehat{M}_\maq)\ge 3$. Therefore we can invoke Lemma \ref{AttachPrimeDepthCanonicalModuleLemma}(ii) to deduce that $\dep\big((\omega_{\widehat{M}})_\maq\big)=2$. Hence, $\dep(\widehat{R}_\maq/\map\widehat{R}_\maq)\le 2$,  by \cite[Proposition 1.2.16]{Herzog} (Note that, $(\omega_M)_\map\ten_{R_\map}\Rc_\maq\cong (\omega_{\Mc})_{\maq}\cong \omega_{\Mc_\maq}$). But then using the fact that $\widehat{R}_\maq/\map\widehat{R}_\maq$ is $S_3$ we can deduce that, the fiber $\widehat{R}_\maq/\map\widehat{R}_\maq$ is Cohen-Macaulay. So, it follows that $\hit_M\map\ge 3$, otherwise $M_\map$ and thence $\Mc_\maq$ is Cohen-Macaulay which is a contradiction. But then we get, $$2\le \min\{2,\hit_{\omega_M}\map\}\le \dep_{R_\map}\big((\omega_M)_\map\big)\le \dep_{\widehat{R}_\maq}\big((\omega_{\widehat{M}})_\maq\big)=2.$$
                  
              (ii)  We, firstly, argue by induction on $d'$. $\omega_M$ is $S_2$ (see Lemma \ref{CanonicalModuleAndModuleHeghtInequality}(iv)). If $d'\le 2$, then both of  $M$  and $\omega_M$ are Cohen-Macaulay and there is nothing to prove. So we assume that $d'\ge 3$. If $n=2$, then the statement follows from the preceding part. Suppose that $n>2$ and the statement has been proved for smaller values of $n$.
              
                Presume that $\omega_M$ is $S_n$. Since $\omega_M$ satisfies the Serre-condition $S_{n-1}$ also, so our inductive hypothesis implies that $\Ht^{i}_\mam(M)=0$ for each $d'-n+3\le i\le d'-1$ and that, $$\At\big(\Ht^{d'-n+2}_\mam(M)\big)=\{\map\in\ts(R):\dep_{R_\map}\big((\omega_M)_\map\big)=n-1\text{\ and\ }, \hit_{\omega_M}\map\ge n\}.$$ But  $\omega_M$ satisfies  $S_n$-condition, so the set on the right hand side of the above  identity is the empty set. Consequently, we get $\Ht^{d'-n+2}_\mam(M)=0$, as claimed. It remains to describe, $\At\big(\Ht^{d'-n+1}_\mam(M)\big)$. In the case where $d'\le n$ we have $\omega_M$ is Cohen-Macaulay and  $d'-n+1\le 1$. Since, $M$ is $S_2$ and $d'\ge 3$, $\dep_R(M)\ge 2$ and consequently $\Ht^0_\mam(M)=\Ht^1_\mam(M)=0.$ So we may assume that $d'\gneq n$, in particular $d'\gneq 3$. Let, $\map\in \At\big(\Ht^{d'-n+1}_\mam(M)\big)$. We, firstly, deal with the case that $\map\neq \mam$. There exists $\maq\in \At_{\widehat{R}}\big(\Ht^{d'-n+1}_{\mam\widehat{R}}(\widehat{M})\big)$ such that $\maq$ lies over $\map$ (see  \cite[11.2.10 Exercise(iii)]{LocalCohomology}). Then \cite[11.2.8. Exercise(i)]{LocalCohomology} together with the identity (\ref{htCodimeIneqality}) implies that $\maq\widehat{R}_\maq\in \At_{\maq\widehat{R}_\maq}\big(\Ht^{\dime_{\widehat{R}_\maq}(\widehat{M}_\maq)-n+1}_{\maq\widehat{R}_\maq}(\widehat{M}_\maq)\big)$. So, our inductive assumption on dimension implies that $\dep_{\widehat{R}_\maq}\big((\omega_{\widehat{M}})_\maq\big)=n$ and   $\dime_{\widehat{R}_\maq}\big((\omega_{\widehat{M}})_{\maq}\big)=\dime_{\widehat{R}_\maq}(\widehat{M}_\maq)\ge n+1$. Therefore, $\dep(\widehat{R}_\maq/\map\widehat{R}_\maq)\le n$ and whence the fiber $\widehat{R}_\maq/\map\widehat{R}_\maq$ is Cohen-Macaulay. If $\dime_{R_\map}\big((\omega_M)_\map\big)\le n$, then  $(\omega_M)_\map$ is Cohen-Macaulay (Recall that $\omega_M$ is assumed to be $S_n$) and hence $(\omega_{\Mc})_\maq$ is Cohen-Macaulay which is a contradiction. Hence $\dime_{R_\map}\big((\omega_M)_\map\big)\ge n+1$. 
                So, $$n\le \min\{n,\dime_{R_\map}\big((\omega_M)_\map\big)\}\le \dep_{R_\map}\big((\omega_M)_\map\big)\le \dep_{\widehat{R}_\maq}\big((\omega_{\widehat{M}})_\maq\big)=n.$$ Thus our statement holds in this case where $\map\neq \mam$.  Suppose that $\map=\mam$. Since $\dime(\omega_M)=\dime(M)\ge n+1$ so it is enough to show that $\dep_R(\omega_M)=n$. Let $x$ be an $M$-regular element. Since, $\Ht^{d'-1}_\mam(M)=0$, so the exact sequence (\ref{SecondExactSequence}) in Remark \ref{FirstRemark} yields, $\omega_{\widehat{M}}/x\omega_{\widehat{M}}\cong \omega_{\widehat{M}/x\widehat{M}}$. We have, $\Ht^{i}_{\mam\widehat{R}}\big(\widehat{M}\big)=0$ for each $d'-n+2\le i\le d'-1$ and $\mam\widehat{R}\in \At_{\widehat{R}}\big(\Ht^{d'-n+1}_{\mam\widehat{R}}(\widehat{M})\big)$. So using the exact sequence, $0\rightarrow \widehat{M}\rightarrow \widehat{M}\rightarrow \widehat{M}/x\widehat{M}\rightarrow 0$ a standard argument shows that, $\Ht^{i}_{\mam\widehat{R}}\big(\widehat{M}/x\Mc\big)=0$ for each $d'-n+2\le i\le d'-2$ and that $\mam\widehat{R}\in \At_{\widehat{R}}\big(\Ht^{d'-n+1}_{\mam\widehat{R}}(\widehat{M}/x\widehat{M})\big)$. Consequently, by   Lemma \ref{Aoyama's result}  we get $\dep_{\widehat{R}}(\omega_{\widehat{M}/x\widehat{M}})=n-1$, i.e. $\dep_{\widehat{R}}(\omega_{\widehat{M}}/x\omega_{\widehat{M}})=n-1$, i.e. $\dep_R(\omega_M)=\dep_{\widehat{R}}(\omega_{\Mc})=n$.
              
              For the reverse inclusion,  let $\map\in \ts(R)$ with $\dep_{R_\map}\big((\omega_M)_\map\big)=n$ and $\hit_{\omega_M}\map\ge n+1.$  Let $\maq\in \text{min}_{\widehat{R}}(\map\widehat{R})$. So by \cite[Proposition 1.2.16]{Herzog} and \cite[Theorem A.11]{Herzog} we have, $\dep_{\widehat{R}_\maq}\big((\omega_{\widehat{M}})_\maq\big)=n$ and $\hit_{\omega_{\widehat{M}}}\maq\ge n+1$. Let $x$ be an $M$-regular element in $\map$. By, Lemma \ref{CanonicalModuleAndModuleHeghtInequality}(i), $x$ is an $\omega_M$-regular element. Thus,  $\dep_{\widehat{R}_\maq}\big((\omega_{\widehat{M}/x\widehat{M}})_\maq\big)=n-1$ and $\hit_{\omega_{\widehat{M}/x\widehat{M}}}\maq\ge n$ (Recall that $\omega_{\Mc/x\Mc}\cong\omega_{\Mc}/x\omega_{\Mc}$). Now, since $\omega_{\widehat{M}/x\widehat{M}}$ satisfies the $S_{n-1}$-condition, so in view of our inductive hypothesis we get $\maq\in \At_{\widehat{R}}\big(\Ht^{d'-n+1}_{\mam\widehat{R}}(\widehat{M}/x\widehat{M})\big)$. Consequently,  since $\Ht^{d'-n+2}_{\mam\Rc}(\Mc)=0$, so we may deduce that $\maq\in \At_{\widehat{R}}\big(\Ht^{d'-n+1}_{\mam\widehat{R}}(\widehat{M})\big)$ which yields $\map\in \At_{R}\big(\Ht^{d'-n+1}_{\mam}(M)\big)$ due to \cite[Theorem 9.5.]{Matsumura} and \cite[11.2.10 Exercise(iii)]{LocalCohomology}.
              
                            Now, conversely, assume that,  $\Ht^i_\mam(M)=0$, for each $d'-n+2\le i\le d'-1$. Then, the inductive assumption shows that, $$\At\big(\Ht^{d'-(n-1)+1}_\mam(M)\big)=\{\map\in\ts(R):\dep_{R_\map}\big((\omega_M)_\map\big)=n-1\text{\ and\ }, \hit_{\omega_M}\map\ge n\}=\emptyset,$$ and that $\omega_M$ satisfies $S_{n-1}$. Suppose by way of contradiction that there exists $\map\in \ts(R)$ such that $\dep_{R_\map}\big((\omega_M)_\map\big)\lneq \min\{n,\hit_{\omega_M}\map\}$. Also, we have, $\dep_{R_\map}\big((\omega_M)_\map\big)\ge \min\{n-1,\hit_{\omega_M}\map\}$. These facts yield $\hit_{\omega_M}\map\ge n$ and $\dep_{R_\map}\big((\omega_M)_\map\big)=n-1$. But then $\map\in \At\big(\Ht^{d'-(n-1)+1}_\mam(M)\big)$ which is a contradiction.

           \end{proof}
     \end{thm}

    \section{Quasi-Gorenstein rings and regular elements}
    
   If  $x\in R\backslash \text{Z}(R)$ then \cite[Lemma 5.1]{CanonicalElements} states that $\omega_R/x\omega_R\cong \omega_{R/xR}$ if and only if $x\notin \unil_{\map\in\At_R\big(\Ht^{d-1}_\mam(R)\big)}\map$. But, as the proof\footnote{The proof is given in the paragraph before the statement of the lemma.} of the Lemma  shows us,   we strongly believe that  the author means    the following  observation: 
   
    The map, $\iota$,  of the exact sequence (\ref{SecondExactSequence}) in Remark \ref{FirstRemark} is an isomorphism if and only if $x\notin \unil_{\map\in\At_R\big(\Ht^{d-1}_\mam(R)\big)}\map$. In fact the proof does not deal with the following question:  
    
    Whether the map, $\iota$, of the exact sequence (\ref{SecondExactSequence}) in Remark \ref{FirstRemark} is an isomorphism whenever $\omega_R/x\omega_R\cong \omega_{R/xR}$? 
    
   At least in the case where $\omega_R=R$, as the following lemma shows us,  the above question has a positive answer.
    
    \begin{prop} \label{quotient_zero_ht}
       Suppose that $R$ is a quasi-Gorenstein ring. Assume that  $x\in \mam$ is a regular element of $R$. Then $R/xR$ is a quasi-Gorenstein ring if and only if $x\notin \unil_{\map\in \text{Att}_R\big(H^{d-1}_\mam(R)\big)}\map$.
         \begin{proof}
             Set $M=R$ in the exact sequence (\ref{SecondExactSequence}) in Remark \ref{FirstRemark}. Suppose that,  $x\notin \bigcup\limits_{\mathfrak{p}\in Att\big(H^{d-1}_{\mathfrak{m}}(R)\big)}\mathfrak{p}$. 
            Then $C=0$ in the exact sequence (\ref{SecondExactSequence}). Thus, $\iota$, is an isomorphism and whence $\omega_{R/xR}\cong \omega_R/x\omega_R\cong R/xR$.
          
          Now, conversely, assume that $R/xR$ is quasi-Gorenstein.  It is harmless to assume that $R$ is complete. Namely the assertion, $x\notin \unil_{\maq\in \At_{\widehat{R}}\big(\Ht^{d-1}_{\mam\widehat{R}}(\widehat{R})\big)}\maq$,  is equivalent to the assertion that $\Ht^{d-1}_{\mam\widehat{R}}(\widehat{R})/x\Ht^{d-1}_{\mam\widehat{R}}(\widehat{R})=0$. But according to  \cite[Flat Base Change Theorem 4.3.2]{LocalCohomology} and faithfully flatness of $\widehat{R}$ the latter statement  holds if and only if $H^{d-1}_{\mathfrak{m}}(R)/xH^{d-1}_{\mathfrak{m}}(R)=0$, that is to say, $x\notin \unil_{\maq\in \At\big(H^{d-1}_{\mathfrak{m}}(R)\big)}\maq$. On the other hand it is a well-known fact that a commutative local ring is quasi-Gorenstein if and only if its completion is so.
          
           Suppose by way of contradiction that $x\in \map$ for some $\map\in\At\big(\Ht^{d-1}_{\mam}(R)\big)$. Since $\Ht^{d-1}_\mam(R/xR)\cong \text{E}_{R/xR}(R/\mam)\cong 0:_{\Ht^{d}_\mam(R)}x$,  the exact sequence (\ref{FirstExactSequence}) in Remark \ref{FirstRemark} gives us an exact sequence,
          $$0\rightarrow \Ht^{d-1}_{\mam}(R)/x\Ht^{d-1}_{\mam}(R)\rightarrow \text{E}_{R/xR}(R/\mam)\overset{f}{\rightarrow}\text{E}_{R/xR}(R/\mam)\rightarrow 0.$$
        Using \cite[10.2.11 Theorem]{LocalCohomology} we can deduce that $f=s\text{id}_{R/xR}{\text{E}(R/\mam)}$ for some $s\in \widehat{R}$. Dualizing  we obtain the exact sequence,
        
        \begin{equation} \label{ThirdExactSequence}
        0\rightarrow R/xR\overset{s\text{id}_{R/xR}}\longrightarrow R/xR\rightarrow \big(\Ht^{d-1}_\mam(R)/x\Ht^{d-1}_\mam(R)\big)^\vee\rightarrow 0. 
        \end{equation}
        
        Since $\map\in \At_R\big(\Ht^{d-1}_\mam(R)\big)$, 
        so a similar argument as in the proof of Lemma \ref{AttachPrimeDepthCanonicalModuleLemma}(i) shows that $\map/xR\in \At_{R/xR}\big(\Ht^{d-1}_\mam(R)/x\Ht^{d-1}_\mam(R)\big).$ Consequently, the exact sequence  (\ref{ThirdExactSequence}) together with \cite[10.2.20 Corollary]{LocalCohomology} implies that $\map/xR\in \As_{R/xR}\big(R/(sR+xR)\big)$. Since $R/xR$ is quasi-Gorenstein so by   Reminder \ref{PreliminariesCanonical}(iv), $R/xR$ satisfies the Serre condition $S_2$. Consequently, in view of Reminder \ref{PreliminariesCanonical}(iii),  the fact that, $\map/xR\in \As_{R/xR}\big(R/(sR+xR)\big)$ yields $\hit_{R/xR}\big(\map/xR\big)=1$,
          i.e. $\hit_R(\map)=2$. Thus, $\dim(R/\map)=d-2$, by virtue of Reminder \ref{PreliminariesCanonical}(iv) and \cite[page 250, Lemma 2]{Matsumura}. Hence we have, $\map\in \At\big(\Ht^{\dim(R/\map)+1}_{\mam}(R)\big)$. Consequently  \cite[11.2.8 Exercise]{Local Cohomology} yields $\map R_\map\in \At\big(\Ht^{1}_{\map R_\map}(R_\map)\big)$. On the other hand, $R$ is $S_2$ and  $\hit_R(\map)=2$, thence $R_\map$ is Cohen-Macaulay and $\Ht^{1}_{\map R_\map}(R_\map)=0$. This contradicts with   $\map R_\map\in \At\big(\Ht^{1}_{\map R_\map}(R_\map)\big)$.

         \end{proof}
    \end{prop}    
     
     The following lemma shall be used in the proof of the corollary \ref{S_2Characterization}(ii).  
     
     \begin{lem} \label{FiniteAlgebraAndInjectiveEnvelop} Assume that $S$ is an $R$-algebra which is finitely generated as $R$-module ($S$ is, indeed, semi-local but not necessarily local). Then, $$\homm_R\big(S,\text{E}(R/\mam)\big)\cong \dsuml_{\man\in\text{Max}(S)}\text{E}(S/\man).$$
       \begin{proof}
          Recall that $S$ is a semi-local ring (see  e.g., \cite[page 69, 9.3.]{Matsumura} and \cite[page 66, Lemma 2]{Matsumura}). Taking into account an epimorphism, $R^n\rightarrow S$ we get a monomorphism $\homm_R\big(S,\text{E}(R/\mam)\big)\rightarrow \dsuml_{i=1}^n\Et(R/\mam)$. It follows that $\homm_R\big(S,\text{E}(R/\mam)\big)$ is an Artinian $R$-module and so  it is an Artinian $S$-module. But $\homm_R\big(S,\text{E}(R/\mam)\big)$ is also an injective $S$-module. Therefore $\homm_R\big(S,\text{E}(R/\mam)\big)\cong \dsuml_{\man\in \text{Max}(S)}\Et(S/\man)^{\mu(\man)}$ (see  \cite[Theorem 3.2.8]{Herzog}). We aim to prove that $\mu(\man)=1$ for each $\man\in \Max(S)$. Let, $\man\in \text{Max}(S)$. Since $S/\man$ is a finite extension of $R/\mam$ so there exists $u\in \mn$ such that $S/\man\cong \dsum\limits_{i=1}^uR/\mam$. We have,
          \begin{align*}
          \dsuml_{i=1}^{\mu(\man)u}R/\mam\cong \dsuml_{i=1}^{\mu(\man)}S/\man&\cong  \homm_S\big(S/\man,\homm_R(S,\Et(R/\mam)\big)&\\ &\cong \homm_R\big(S/\man,\Et(R/\mam)\big)&\\ &\cong \homm_R\big(\dsuml_{i=1}^{u}R/\mam,\Et(R/\mam)\big)&\\ &\cong \dsuml_{i=1}^{u}R/\mam.
          \end{align*}
          It turns out that $\mu(\man)=1$ as required.
       \end{proof}
     \end{lem}
     
     The subsequent lemma is required for the next corollary.
     
     \begin{lem}\label{LiftingPropertyOfIsomorphisms}
       Let $S$ be an $R$-algebra which is  finitely generated as $R$-module. The following statements hold.
         \begin{enumerate}
           \item[(i)] Assume that, the natural map $\varphi:\Ht^d_\mam(R)\rightarrow \Ht^d_\mam(R)\tenr S$,  given by $\alpha\mapsto \alpha\ten 1$ is an $R$-isomorphism. Then  given an $S$-module $N$ and a map $f:\Ht^d_{\mam S}(S)\rightarrow \Ht^d_{\mam S}(N)$ we have  $f$ is an $R$-homomorphism if and only if it is an $S$-homomorphism.
           \item[(ii)] Suppose that $R$ has a canonical module $\omega_R$. For each $d$-dimensional finitely generated $S$-module $N$ the well-known $R$-isomorphism, $$\homm_R\big(\homm_R(N,\omega_R),\text{E}(R/\mam)\big)\cong \Ht^d_{\mam S}(N),$$ is an $S$-isomorphism too.
         \end{enumerate}
         \begin{proof}
           (i) For each finite $S$-module $N$ by representing  $\Ht^d_\mam(R)$, (respectively $\Ht^d_{\mam S}(N)$) as the last cohomology of the \quad \v{C}ech complex $C^\bullet(\mathbf{x},R)$ (respectively, $C^\bullet(\mathbf{x}S,N)$) for some system of parameters $\mathbf{x}$ for $R$, it is easily seen that the natural $R$-isomorphism, $$\psi_N:\Ht^d_\mam(R)\tenr N\rightarrow \Ht^d_{\mam S}(N),$$ which takes an element, $[r/\mathbf{x}^\alpha]\tenr n\in \Ht^d_\mam(R)\tenr N$,  to, $[rn/\mathbf{x}^\alpha]$, is an $S$-isomorphism too. According to our hypothesis, we have the $R$-isomorphism, $\psi_S\circ\varphi:\Ht^d_\mam(R)\rightarrow \Ht^d_{\mam S}(S)$. Therefore, 
             \begin{align*}
               \homm_S\big(\Ht^d_{\mam S}(S),\Ht^d_{\mam S}(N)\big)&\overset{\homm(\psi_S,\text{id})}{\cong} \homm_S\big(\Ht^d_{\mam }(R)\tenr S,\Ht^d_{\mam S}(N)\big)\cong \homm_R\Big(\Ht^d_{\mam }(R),\homm_S\big(S,\Ht^d_{\mam S}(N)\big)\Big)&\\ & \ \ \ \ \ \ \ \cong\ \homm_R\big(\Ht^d_{\mam }(R),\Ht^d_{\mam S}(N)\big) \overset{\homm(\varphi^{-1}\circ\psi_S^{-1},\text{id})}{\cong} \homm_R\big(\Ht^d_{\mam S}(S),\Ht^d_{\mam S}(N)\big).
             \end{align*}
             Let us to denote the composition of the above isomorphisms with $\lambda$. It is easily seen that $\lambda(f)=f$ for each $f\in \homm_S\big(\Ht^d_{\mam S}(S),\Ht^d_{\mam S}(N)\big)$. This fact together with the surjectivity of $\lambda$ proves the claim.
             
             (ii) Let $N$ be a finitely generated $S$-module. As we have seen just in the previous part  there exists a natural $S$-isomorphism, $\psi_N:\Ht^d_\mam(R)\tenr N\rightarrow \Ht^d_{\mam S}(N).$ Hence we have the following chain  of $S$-isomorphisms,
               $$ \homm_R\big(\homm_R(N,\omega_R),\text{E}(R/\mam)\big)\cong \homm_R\big(\omega_R,\text{E}(R/\mam)\big)\tenr N\cong \Ht^d_\mam(R)\tenr N\cong \Ht^d_{\mam S}(N).$$ 
             (The reader could easily verify that the first isomorphism is an $S$-isomorphisms too)
         \end{proof}
     \end{lem}

     \begin{cor} \label{S_2Characterization} Assume that $R$ is a quasi-Gorenstein ring and $x\in \mam$ is a regular element. Then the following statements hold.
     \begin{enumerate}
       \item[(i)] A necessary and sufficient condition for $R/xR$ to be quasi-Gorenstein is that $R/xR$ satisfies the $S_2$ condition.
       \item[(ii)] If $R/xR$ has a canonical module then  the $S_2$-ification of $R/x R$ is quasi-Gorenstein (up to localization).      
     \end{enumerate}
       
       \begin{proof}

          (i)  Since, $R$ is quasi-Gorenstein so $\Rc$ is quasi-Gorenstein and hence $S_2$. Hence the formal fibers of $R$ are $S_2$. Therefore the formal fibers of $R/xR$ are $S_2$ too. It turns out that $R/xR$ satisfies $S_2$ precisely when  $\Rc/x\Rc$ is $S_2$.  Therefore, the statement follows from Theorem \ref{CanonicalModuleRegularElement}.
         
         (ii) Let us  denote the $S_2$-ification $\text{Hom}_{R/xR}(\omega_{R/xR},\omega_{R/xR})$ of $R/xR$ by $S$.  Recall that, $S$ is semi-local. Let $\mam_S$ be the Jacobson radical of $S$. We aim to prove that, $\Ht^{d-1}_{\mam_S}(S)\cong \dsum\limits_{\man\in \Max(S)}\Et(S/\man)$. Then the statement follows after localization. In order to accomplish this, we show that $\Ht^{d-1}_{\mam_S}(S)\cong \Ht^{d-1}_{\mam_S}(\omega_{R/xR})$ as $S$-modules. On the other hand using Lemma \ref{FiniteAlgebraAndInjectiveEnvelop} we get $R/xR$-isomorphisms,
         
        \begin{center}
             $\dsuml_{\man\in\text{Max}(S)}\text{E}(S/\man)\cong \homm_{R/xR}\big(S,\text{E}_{R/xR}(R/\mam)\big)= \homm_{R/xR}\big(\homm_{R/xR}(\omega_{R/xR},\omega_{R/xR}),\text{E}_{R/xR}(R/\mam)\big)\cong \text{H}^{d-1}_{\mam_S}(\omega_{R/xR}),$
          \end{center}
          
           which is an $S$-isomorphism too by Lemma \ref{LiftingPropertyOfIsomorphisms}(ii). So, $\Ht^{d-1}_{\mam S}(S)\cong \dsum_{\man\in\Max(S)}\text{E}(S/\man)$, as was to be proved.
           
           Let us to  denote the natural ring homomorphism $R/xR\rightarrow S$, defined by the rule $y\mapsto y\text{id}_{\omega_{R/xR}}$, with $g$.  Let $\mathbf{x}$ be a system of parameters for $R/xR$ whose image in $S$, also denoted by $\mathbf{x}$, is a system of parameters for $S$.  In the following we will prove that not only $\Ht^{d-1}_\mam(g)$  is an isomorphism (see (\ref{EquationFirstLocalCohomology}))  but also $\Ht^{d-1}_\mam(R/xR)\cong \Ht^{d-1}_\mam(\omega_R/x\omega_R)$ (see (\ref{EquationSecondLocalCohomology})). It turns out that $\Ht^{d-1}_\mam(S)\cong \Ht^{d-1}_\mam(\omega_{R/xR})$ as $R/xR$-module. Consider the isomorphism, $\Psi_S:\Ht^{d-1}_\mam(R/xR)\tenr S\rightarrow \Ht^{d-1}_\mam(S)$, which maps an element, $[a/(x_1\ldots x_{d-1})^n]\tenr s\in \Ht^{d-1}_\mam(R/xR)\tenr S$ to $[as/(x_1\ldots x_{d-1})^n]$.  If we take into account the map  $\varphi$ of preceding lemma, then an easy verification  shows that $\Ht^{d-1}_\mam(g)=\psi_S\circ \varphi$. This implies that $\varphi$ is bijective. Hence  $\Ht^{d-1}_\mam(S)\cong \Ht^{d-1}_\mam(\omega_{R/xR})$ as $S$-module by virtue of Lemma \ref{LiftingPropertyOfIsomorphisms}(i). This gives us our desired isomorphism, $$\Ht^{d-1}_{\mam_S}(S)\cong \Ht^{d-1}_{\mam_S}(\omega_{R/xR})$$ of the preceding paragraph and completes our proof.
          
         Since $R$ is quasi-Gorenstein so Lemma \ref{CanonicalModuleAndModuleHeghtInequality}(ii) implies that  $\hit(\map)+\dime(R/\map)=\dime(R)$ for each $\map\in\ts(R)$. Furthermore  $R$ satisfies $S_2$. Thus   Reminder  \ref{PreliminariesCanonical}(iii) yields $\hit(\map)=1$ for each $\map\in\text{Ass}(R/xR)$. It follows that $\dime(R/\map)=d-1$ for each $\map\in \As(R/xR)$, i.e. $R/xR$ is unmixed.  Therefore the identity (\ref{AnnihilatorCanonicalModule}) of Reminder \ref{PreliminariesCanonical}(iv) shows that, $0:_{R/xR}\omega_{R/xR}=0$. Thus   $S$ is a finitely generated integral extension of $R/xR$ by the natural map $g$. Since $R$ is quasi-Gorenstein, so by virtue of Remark \ref{FirstRemark}(c) and Reminder \ref{PreliminariesCanonical}(iv), the exact sequence (\ref{SecondExactSequence}) of Reminder \ref{FirstRemark} gives us an exact sequence,
          $\mathscr{F}:=0\rightarrow \widehat{R}/x\widehat{R}\overset{f}{\rightarrow} \omega_{\widehat{R}/x\widehat{R}}\rightarrow C\rightarrow 0,$ where,
          \begin{center}
          \begin{equation}
             \label{CLocal}
             C_\map\cong \homm_{\Rc_{\map}}\big(\text{H}^{\text{ht}(\map)-1}_{\map \widehat{R}_\map}(\widehat{R}_\map)/x\text{H}^{\text{ht}(\map)-1}_{\map \widehat{R}_\map}(\widehat{R}_\map),\Et_{\Rc_{\map}}(\Rc_{\map}/\map \Rc_{\map})\big)  
          \end{equation}
          \end{center}   
              for each  $\map\in \ts(\widehat{R})$. On the other hand we have the exact sequence, $\mathscr{G}:=0\rightarrow R/xR\overset{g}{\rightarrow} S\rightarrow C^\prime\rightarrow 0.$ Tensoring $\mathcal{G}$ with $\widehat{R}$ yields the exact sequence,
           $$\widehat{\mathscr{G}}:0\rightarrow \widehat{R}/x\widehat{R}\overset{\widehat{g}}{\rightarrow} \homm_{\widehat{R}}(\omega_{\widehat{R}/x\widehat{R}},\omega_{\widehat{R}/x\widehat{R}})\rightarrow C^\prime\tenr \widehat{R}\rightarrow 0.$$
          A similar argument as above shows that $\Rc/x\Rc$ is unmixed and thence $(\omega_{\Rc/x\Rc})_\maq\cong \omega_{(\Rc/x\Rc)_\maq}$ for each $\maq\in \ts(\Rc/x\Rc)$ (see  Reminder \ref{PreliminariesCanonical}(v)).  In light of Proposition \ref{quotient_zero_ht} and (\ref{CLocal})in conjunction with the first part of the corollary  we can deduce that $\map/x\widehat{R}\notin \Supp_{\widehat{R}/x\widehat{R}}(C)$  if and only if $\widehat{R}_\map/x\widehat{R}_\map$ satisfies the $S_2$ condition. On the other hand by \cite[Proposition 1.2]{Aoyama2} we have $\widehat{g}_{\map/x\widehat{R}}$ is an isomorphism  if and only if $\widehat{R}_\map/x\widehat{R}_\map$ satisfies $S_2$.  It turns out that, $$\dime_{\widehat{R}/x\widehat{R}}(C)=\dime_{\widehat{R}/x\widehat{R}}(C^\prime\tenr \widehat{R})=\dime_{R/xR}(C^\prime).$$
          
          Since $R/xR$ is $S_1$, so $R/xR$ is locally Cohen-Macaulay at each  prime ideal of height less than or equal one, whence by \cite[Proposition 1.2]{Aoyama2} $\hit_{R/xR}(0:_{R/xR}C^\prime)\ge 2$ and thereby, $$\text{dim}_{R/xR}(C^\prime)=\dime\big((R/xR)/(0:_{R/xR}C^\prime)\big)\le \dime(R/xR)-\hit(0:_{R/xR}C^\prime)\le d-3.$$  Thus using the exact sequence $\mathscr{G}$ we conclude that 
          \begin{equation}\label{EquationFirstLocalCohomology}
            \text{H}^{d-1}_\mam(S)\cong \text{H}^{d-1}_\mam(R/xR). 
          \end{equation}            
             Since $\dime(C)=\dime(C^\prime)$, using $\mathscr{F}$, we deduce that $\text{H}^{d-1}_\mam(\widehat{R}/x\widehat{R})\cong \text{H}^{d-1}_\mam(\omega_{\widehat{R}/x\widehat{R}})$ similarly. This implies that,
           \begin{equation}\label{EquationSecondLocalCohomology}
             \text{H}^{d-1}_\mam(R/xR)\cong \text{H}^{d-1}_\mam(\widehat{R}/x\widehat{R})\cong \text{H}^{d-1}_\mam(\omega_{\widehat{R}/x\widehat{R}})\cong \text{H}^{d-1}_\mam(\omega_{R/xR}).
             \end{equation}
             
       \end{proof}
     \end{cor}
     
     The first part of the preceding corollary is a special case of the following theorem.
     
     \begin{thm}\label{CanonicalModuleRegularElement}
     	Suppose that $\widehat{R}$ satisfies the $S_2$-condition and that $x$ is a regular element of $R$.  Then, $\omega_{\widehat{R}}/x\omega_{\widehat{R}}\cong \omega_{\widehat{R}/x\widehat{R}}$ if and only if $\omega_{\widehat{R}}/x\omega_{\widehat{R}}$ satisfies the Serre-condition $S_2$.
     	  \begin{proof}
     	  	  We prove that $\omega_{\widehat{R}/x\widehat{R}}$ is isomorphic to  $\omega_{\widehat{R}}/x\omega_{\widehat{R}}$ provided the latter  is $S_2$. The converse is obvious by By Remark \ref{PreliminariesCanonical}(v). Suppose, to the contrary, that the map, $\iota,$ in the  exact sequence (\ref{SecondExactSequence}) of Remark \ref{FirstRemark}, for $M:=R$, is not isomorphism, i.e., $\text{coker}(\ \iota)=C=\big(\Ht^{d-1}_\mam(R)/x\Ht^{d-1}_\mam(R)\big)^\vee\neq 0$. It follows that, $x\in \unil_{\map\in\At\big(\Ht^{d-1}_{\mam}(R)\big)}\map,$ and thence,  $x\in \unil_{\map\in\At\big(\Ht^{d-1}_{\mam}(\widehat{R})\big)}\map$. So, by  Theorem \ref{AttPrimeSnLocalCohomology}(i), there exists a prime ideal $\maq\in \ts(\widehat{R})$, containing $x$, such that $\dep_{\widehat{R}_\maq}\big((\omega_{\widehat{R}})_\maq\big)=2$ but $\hit_{\omega_{\widehat{R}}}\maq\ge 3$. Consequently, $\dep_{\widehat{R}_\maq}\big((\omega_{\widehat{R}}/x\omega_{\widehat{R}})_\maq\big)=1$ and $\hit_{\omega_{\widehat{R}}/x\omega_{\widehat{R}}}\maq\ge 2$, which violates the $S_2$-property of $\omega_{\widehat{R}}/x\omega_{\widehat{R}}$. 
     	  \end{proof}
     \end{thm}
     
     The first author of the paper, in \cite[Proposition 3.4.]{TavanfarReduction}, shows that the Ulrich's deformation of certain Gorenstein rings to unique factorization domains, developed in \cite{UlrichGorenstein },  has a quasi-Gorenstein counterpart. We will mention the statement of this deformation theorem, for the sake of completeness of the paper.
     
     \begin{thm} (see, \cite[Proposition 3.4.]{TavanfarReduction} and \cite[Proposition 1]{UlrichGorenstein})
     	 \label{QuasiGorensteinDeformsToUFD} Suppose that $P$ is a regular  ring and $R:=P/\maa$ is a quasi-Gorenstein ring. Assume, furthermore, that $R$ is locally complete intersection at codimension $\le 1$. Then there exists a unique factorization domain $S$ (which is of finite type over $P$) and a regular sequence $\mathbf{y}$ of $S$ such that $R\cong S/(\mathbf{y})$.
     \end{thm}
     
     In spite of the above deformation theorem, at the of writing this paper, it is not clear for us whether the quasi-Gorenstein property deforms. However, in view of the following theorem,     the quasi-Gorenstein property adheres a variant of deformation which we call it as analytic deformation. The following theorem will be used in the proof of Theorem \ref{G-dim}.

     \begin{thm} \label{QuasiGorensteinDeformation} Suppose that $x\in R\backslash\text{Z}(R)$ and $R/x^nR$ is quasi-Gorenstein for each $n\in \mn$. Then $R$ is quasi-Gorenstein.
       \begin{proof}
        Our proof reduces to the complete case, so there exists a Gorenstein local ring $(S,\man)$ such that $\dime(S)=d$ and $R$ is a homomorphic image of $S$. The commutative diagrams,
           \begin{equation}\label{FirstDiagram}\begin{CD}
             0@>>> R @>x^n>> R @>>> R/x^nR@>>>0\\
             @. @| @VVxV @VVf_n:=(r+x^{n}R\mapsto rx+x^{n+1}R)V \\
             0@>>> R @>x^{n+1}>> R @>>> R/x^{n+1}R@>>>0
           \end{CD},\end{equation}
          leads to the commutative diagrams,
          \begin{equation}\label{SecondDiagram}\begin{CD}
             0@>>> \homm_S(R,S)/x^n\homm_S(R,S) @>>> \ext^1_S(R/x^nR,S) @>>> 0:_{\ext^1_S(R,S)}x^n@>>>0\\
             @. @AAA @AA\ext^1_S(f_n,\text{id}_S)A @AAxA \\
             0@>>> \homm_S(R,S)/x^{n+1}\homm_S(R,S) @>>> \ext^1_S(R/x^{n+1}R,S) @>>> 0:_{\ext^1_S(R,S)}x^{n+1}@>>>0
           \end{CD},\end{equation}
           for each $n\in \mn$. Taking the inverse limit we get the exact sequence, 
           \begin{equation}
             \label{ExactsequenceCanonicalModule}
             0\rightarrow\lim\limits_{\underset{n\in\mn}{\longleftarrow}}\omega_R/x^n\omega_R\rightarrow \lim\limits_{\underset{n\in\mn}{\longleftarrow}}\ext^1_S(R/x^nR,S)\rightarrow \lim\limits_{\underset{n\in\mn}{\longleftarrow}}0:_{\ext^1_S(R,S)}x^n.
           \end{equation}
            But, $\omega_R\cong \lim\limits_{\underset{n\in\mn}{\longleftarrow}}\omega_R/x^n\omega_R$, by \cite[Theorem 8.7]{Matsumura}    and \cite[page 63, 8.2.]{Matsumura}. By our hypothesis we have,
            \begin{align*}
               \ext^1_S(R/x^nR,S)&\cong \homm_S\big(\Ht^{d-1}_\man(R/x^nR),\text{E}_S(S/\man)\big)\cong \homm_S\big(\Ht^{d-1}_\mam(R/x^nR)\ten_{R/x^nR} R/x^nR, \Et_S(S/\man)\big)&\\ &\cong \homm_{R/x^nR}\bigg(\Ht^{d-1}_\mam(R/x^nR),\homm_S\big(R/x^nR,\Et_S(S/\man)\bigg)&\\ &\cong \homm_{R/x^nR}\big(\Ht^{d-1}_\mam(R/x^nR),\Et_{R/x^nR}(R/\mam)\big)\\& \cong R/x^nR,
            \end{align*}               
                for each $n\in \mn$.  Hence the inverse system, $$\big(\lim\limits_{\underset{n\in\mn}{\longleftarrow}}\ext^1_S(R/x^nR,S),\ext^1_S(f_n,\text{id}_S)\big)_{n\in\mn},$$ is isomorphic to an inverse system $(R/x^nR,\gamma_n:R/x^{n+1}R\rightarrow R/x^nR)_{n\in \mn}$. For each $n\in \mn$ there exists $\alpha_n\in R$ such that $\gamma_n(r+x^{n+1}R)=\alpha_nr+x^nR$. We claim that all but finitely many of  $\alpha_i's$ are invertible. In fact, one can  easily verified that if our claim is wrong then for an arbitrary element $(r_n+x^nR)_{n\in\mn}\in \lim\limits_{\underset{n\in\mn}{\longleftarrow}}R/x^nR$, we get $r_n+x^nR\in \bigcap\limits_{l\in\mn}(\mam/x^nR)^l=0$ for each $n\in\mn$. Consequently $\lim\limits_{\underset{n\in\mn}{\longleftarrow}}R/x^nR=0$ and thereby $\omega_R=0$ by the exact sequence (\ref{ExactsequenceCanonicalModule}) which is a contradiction. Hence, without loss of generality we can assume that all of $\alpha_i's$ are invertible. Modify the diagram (\ref{FirstDiagram}) so that $x^{n+1}$ is replaced by $\alpha_{n}^{-1}x^{n+1}$ in the lower row and the vertical maps $f_n$ and $xid_R$ are replaced by $\alpha_n^{-1}f_n$ and $\alpha_n^{-1}x$, respectively. This yields a modification of the diagram (\ref{SecondDiagram}) in which $\ext^1_S(f_n,\text{id}_S)$ is replaced by $\alpha_n^{-1}\ext^1_S(f_n,\text{id}_S)$ and, by a slight abuse of notation, $x\text{id}_{{\ext^1_S(R,S)}}$ is replaced by $\alpha_n^{-1}x\text{id}_{\ext^1_S(R,S)}$. The result is an exact sequence analogous to the exact sequence (\ref{ExactsequenceCanonicalModule}) in which the second non-zero module is isomorphic to $\lim\limits_{\underset{n\in\mn}{\longleftarrow}}R/x^nR\cong R$.  
                
                Moreover, $\lim\limits_{\underset{n\in\mn_n}{\longleftarrow}}0:_{\ext^1_S(R,S)}x^n=0,$  by the following fact. The ascending chain of modules, $\{0:_{\ext^1_S(R,S)}x^n\}_{n\in \mn},$ stabilizes eventually, so that the multiplication by a fixed sufficiently large power of $x$ is zero. This concludes our proof.
       \end{proof}
     \end{thm}
     
          \section{Quasi-Gorenstein rings and the $G$-dimension of the canonical module}

    \begin{defi} \cite{BeyondTotallyReflexive} A finitely generated $R$-module $N$ is said to be a totally reflexive $R$-module if it satisfies the following conditions.
     \begin{enumerate}
       \item[(i)] $N$ is reflexive, i.e. the natural evaluation map, $\eta_N:N\rightarrow \homm_R\big(\homm_R(N,R),R\big)$, is an isomorphism.
       \item[(ii)] $\ext^i_R(N,R)=\ext^i_R\big(\homm_R(N,R),R\big)=0$ for each $i\ge 1$.
     \end{enumerate}
        \end{defi}

   \begin{defi} \cite{BeyondTotallyReflexive} An (augmented)  $G$-resolution of a finitely generated $R$-module $N$ is an exact sequence, $\cdots\rightarrow G_i\rightarrow G_{i-1}\rightarrow \cdots\rightarrow G_0\rightarrow N\rightarrow 0,$  where $G_i$ is totally reflexive. Moreover, the $G$-dimension of $N$, denoted by $\text{G-dim}_RN$, is the least integer $n\ge 0$ such that there exists a $G$-resolution of $N$ with $G_i=0$ for each $i>n$ (If there does not exist such a resolution then we say that $\text{G-dim}_RN=\infty$).
   \end{defi}

   \begin{lem}\label{ZeroDepthTotallyReflexive} Suppose that $(R,\mam)$ is a zero depth local ring and $M$ is a totally reflexive $R$-module. If $M$ is indecomposable and $0:_RM=0$, then $M\cong R$. In particular, if $\homm_R(M,M)\cong R$, then $M\cong R$.
     \begin{proof}
       
We assume that $M$ is not free and we look for a contradiction. In particular, $M$ has no non-zero free direct summand. Let $P_{\bullet}$
be a free resolution of $M$ and $F_{\bullet}$ be a free resolution
of $M^{*}$. Then according to the totally reflexiveness of $M$ we
get a complete resolution, $$\ldots\rightarrow P_{2}\rightarrow P_{1}\rightarrow P_{0}\overset{A}{\rightarrow}F_{0}^{*}\rightarrow F_{1}^{*}\rightarrow F_{2}^{*}\rightarrow\ldots,$$
of $M$ where $\text{im}(A)=\text{im}(M\rightarrow F_{0}^{*})$. If
the matrix $A$ has some entries in $R\backslash\mathfrak{m}$, then
there exists an element $x:=(r_{1},\ldots,r_{i-1},1,r_{i+1},\ldots,r_{l})\in\text{im}(A)=\text{im}(M\rightarrow F_{0}^{*})$.
Let $\pi:F_0^*\rightarrow R$ be the projection map onto the $i$-th component. Then, $\pi\rceil_{\text{im}(A)}\in \homm_R\big(\text{im}(A),R\big)$, so $\mathcal{O}_{\text{\text{im}(A)}}(x)=R$, where $\mathcal{O}_{\text{\text{im}(A)}}(x)=\{f(x):f\in \homm_R\big(\text{im}(A),R\big)\}$. Therefore
by \cite[Lemma 9.5.1.]{Herzog}, $M\cong\text{im}(M\rightarrow F_{0}^{*})=\text{im}(A)$,
has a non-zero free direct summand which is a contradiction. It turns
out that  entries of $A$ are elements of $\mathfrak{m}$\footnote{In general, the minimal free resolutions of $M$ and $M^*$ for some totally reflexive $R$-module $M$, similarly as above, give us a matrix $A$ such that its entries  may lie in $R\backslash \mam$. For example, set $M:=R$. Then it is easily seen that the given matrix $A$ will be the matrix $[1]$.}. This immediately
yields, $0:_{R}\mathfrak{m}\subseteq0:_{R}\text{im}(A)=0:_{R}\text{im}(M\rightarrow F_{0}^{*})=0:_{R}M=0$,
which contradicts with the fact that $\text{depth}(R)=0$. It follows
 that $M$ is free, as claimed. 

   If, $\homm_R(M,M)\cong R$, then $0:_{R}M\subseteq0:_{R}\text{Hom}_{R}(M,M)=0:_{R}R=0$. Furthermore, $M$ is indecomposable, otherwise, $R\cong \homm_R(M,M)$ would be decomposable which is a contradiction.
     \end{proof}
   \end{lem}
   
   \begin{rem} Let  $M$ be an $R$-module such that $\homm_R(M,M)\cong R$ and  $\ext^1_R(M,M)=0$. Let $x_1\in R\backslash \text{Z}(M)$. Then $\text{Hom}_{R/x_1R}(M/x_1M,M/x_1M)\cong R/x_1R$. If in addition, $\ext^2_R(M,M)=0$, and $x_1,x_2$ is an $M$-regular sequence  such that $x_1\notin\text{Z}(R)$ then by \cite[page 140, Lemma 2.(i)]{Matsumura}], $\ext^1_{R/x_1R}(M/x_1M,M/x_1M)\cong \ext^2_R(M/x_1M,M)$. On the other hand from  the vanishing of the modules $\ext^1_R(M,M)$ and $\ext^2_R(M,M)$ we can deduce that $\ext^2_R(M/x_1M,M)=0$ and whence $\ext^1_{R/x_1R}(M/x_1M,M/x_1M)=0$. Thus, again, by the same argument as above we obtain,
   
    $$\homm_{R/\big((x_1,x_2)R\big)}\Big(M/\big((x_1,x_2)M\big),M/\big((x_1,x_2)M\big)\Big)\cong R/\big((x_1,x_2)R\big).$$
   \end{rem}

   \begin{thm}\label{G-dim}
     Assume that $R$ has a canonical module $\omega_R$ and that $\text{Gdim}(\omega_R)<\infty$. Then $R$ is quasi-Gorenstien provided $\ext^i_R(\omega_R,\omega_R)=0$ for each $1\le i\le \dep(\omega_R)$.      
       \begin{proof}
         We induct on $\dime(R)$. In the case where $\dime(R)=0$ the statement follows from  Lemma \ref{ZeroDepthTotallyReflexive} and the Auslander-Bridger formula \cite[Theorem 1.25.]{BeyondTotallyReflexive}. So we assume that $\dime(R)>0$ and the statement has been proved for smaller values of $\dime(R)$. Applying the Auslander-Bridger formula \cite[Theorem 1.25.]{BeyondTotallyReflexive} we get, $$\dep(R_\map)=\dep(\omega_{R_\map})+\text{Gdim}(\omega_{R_\map})\ge \dep_R(\omega_{R_\map})\ge \min\{2,\dime(R_\map)\},$$ for each $\map\in \Supp_R(\omega_R)$. In particular, by virtue of \cite[Proposition 1.2]{Aoyama2} we can conclude that $R$ is $S_2$ and thence $\homm_R(\omega_R,\omega_R)\cong R$.
         
         Let, $\dep(\omega_R)=2$ (respectively, $\dep(\omega_R)=1$). Since $\dep(R)\ge \dep(\omega_R)$ so there exists a regular sequence $\mathbf{x}:=x_1,x_2$ (resp. $\mathbf{x}:=x_1$) on $R$. By \cite[1.10]{Aoyama}, $\mathbf{x}$ is also a regular sequence on $\omega_R$. The above remark in conjunction with our hypothesis implies that $$\homm_{R/(\mathbf{x})}\big(\omega_{R}/\mathbf{x}\omega_R,\omega_{R}/\mathbf{x}\omega_R\big)\cong R/(\mathbf{x}).$$

             So, 
          \begin{align*}\mam/(\mathbf{x})\in \As_{R/\mathbf{x}R}(\omega_R/\mathbf{x}\omega_R)&=\As_{R/\mathbf{x}R}(\omega_R/\mathbf{x}\omega_R)\bigcap \Supp_{R/\mathbf{x}}(\omega_R/\mathbf{x}\omega_R)&\\&=\text{Ass}_{R/\mathbf{x}R}\big((\homm_{R/(\mathbf{x})R}(\omega_R/\mathbf{x}\omega_R,\omega_R/\mathbf{x}\omega_R)\big)&\\&=\As_{R/\mathbf{x}R}\big(R/\mathbf{x}R\big).
          \end{align*}   
            Consequently, $\dep\big(R/\mathbf{x}R\big)=0$. Thus, Lemma \ref{ZeroDepthTotallyReflexive} together with \cite[Corollary 1.4.6]{Christensen}, implies that $\omega_R/\mathbf{x}\omega_R\cong R/(\mathbf{x})$. Thus, $\mu(\omega_R)=\mu(\omega_R/\mathbf{x}\omega_R)=1$. Also by \cite[Lemma 1.1]{Aoyama2} together with (\ref{AnnihilatorCanonicalModule}) in Reminder \ref{PreliminariesCanonical}, we have $0:_R\omega_R=0$. Hence $\omega_R\cong R$.
            
             Now we deal with the case where $\dep(\omega_R)\ge 3$.  We may assume that $R$ is complete. In this case, Lemma \ref{AttachPrimeDepthCanonicalModuleLemma}(ii) implies that $\mam\notin \text{Att}\big(\text{H}^{d-1}_\mam(R)\big)$. Let, $x\in R\backslash \big(\text{Z}(R)\bigcup (\unil_{\map\in \text{Att}\big(\text{H}^{d-1}_\mam(R)\big)}\map)\big)$.   Then the exact sequence (\ref{SecondExactSequence}) in Remark \ref{FirstRemark} shows that $\omega_R/x\omega_R$ is the canonical module of $R/xR$.  Thus using \cite[Corollary 1.4.6]{Christensen} we can conclude that  $\text{Gdim}_{R/xR}(\omega_{R/xR})<\infty$.  Our hypothesis in conjunction with \cite[page 140, Lemma 2.(i)]{Matsumura} implies that $\ext^i_{R/xR}(\omega_{R/xR},\omega_{R/xR})=0$  for each $1\le i\le \dep_{R/xR}(\omega_{R/xR})$. Therefore  using the inductive assumption we can deduce that  $R/xR$ is quasi-Gorenstein. By the same token, $R/x^nR$ is quasi-Gorenstein for each $n\in \mn$. Thus, by virtue of Theorem \ref{QuasiGorensteinDeformation} we conclude that $R$ is quasi-Gorenstein.
       \end{proof}
   \end{thm}
   
   \begin{rem}
      At this time we do not know whether  the finiteness of the $G$-dimension of $\omega_R$ implies that $\ext^1_R(\omega_R,\omega_R)=\ext^2_R(\omega_R,\omega_R)=0$.  Note that in light of  \cite[Theorem 3]{Aoyama3},  this is, indeed, the case when $\text{Projdim}_R(\omega_R)<\infty$. We stress that, by arguing as in the proof of  Theorem \ref{G-dim}, an affirmative answer to this question  shows that $R$ is quasi-Gorenstein precisely when $\text{Gdime}(\omega_R)<\infty$ (c.f.  the hypothesis of Theorem \ref{G-dim}).  Hence, it is perhaps worthwhile to propose the following questions.
   \end{rem}
   
   \begin{ques} \label{ExtQuestion} Suppose that (the local  ring) $R$ has a canonical module with finite $G$-dimension. Then, do we have $\ext^1_R(\omega_R,\omega_R)=\ext^2_R(\omega_R,\omega_R)=0$?
   \end{ques}
   
   \begin{ques} \label{GQuestion} Assume that (the local  ring) $R$ has a canonical module with finite $G$-dimension. Is $R$ a quasi-Gorenstein ring?
   \end{ques}
   
      We stress again that an affirmative answer to the former question would answer the latter question positively. 
 
   The following remark sheds some light on the question of vanishing of   $\ext^1_R(\omega_R,\omega_R)$ and $\ext^2_R(\omega_R,\omega_R)$ which is related to the  question \ref{GQuestion}.

   \begin{rem}  The following statement holds. 
      \begin{enumerate} 
      
        \item[(i)] Suppose that $R$ is a homomorphic image of a Gorenstein local ring $S$ with $\dime(R)=\dime(S)$. Then  we can invoke the Grothendieck spectral sequence, $$\ext^i_R\big(\omega_R,\ext^j_S(R,S)\big)\overset{i,j}{\Rightarrow}\ext^{i+j}_S(\omega_R,S),$$
       and \cite[Theorem 10.33 (Cohomology Five-Term Exact Sequence)]{Rotman} to obtain the exact sequence,
        \begin{center}
        \begin{equation}
          \label{ExactSequence4}
          \Ht^{d-2}_\mam(\omega_R)\rightarrow \ext^2_R(\omega_R,\omega_R)^\vee\rightarrow \Ht^{d-1}_\mam(R)\tenr \omega_R\rightarrow \Ht^{d-1}_\mam(\omega_R)\rightarrow \ext^1_R(\omega_R,\omega_R)^\vee\rightarrow 0.
        \end{equation}
      \end{center}
        \item[(ii)] It is, perhaps, worthwhile to give an example of a ring satisfying $S_2$ but $\ext^1_R(\omega_R,\omega_R)\neq 0$.  By virtue of \cite[Theorem (1.1)]{GotoBuchsbaum} there exists a $5$-dimensional  Quasi-Buchsbaum ring $R$ such that, $\begin{cases}\Ht^i_\mam(R)=0,& i\neq  2,3\\ \Ht^i_\mam(R)\cong R/\mam, & i=2,3\end{cases}$, and it is a homomorphic image of a Gorenstein ring $S$. In particular  since $R$ is Cohen-Macaulay at punctured spectrum and $\dep(R)=2$ so $R$ satisfies the $S_2$-condition. Hence by \cite[Remark 1.4]{Aoyama2} $\Ht^d_\mam(\omega_R)\cong \Et(R/\mam)$ and thence $R=\omega_{\omega_R}$. But it is not $S_3$. Thus in view of  Theorem \ref{AttPrimeSnLocalCohomology}(ii) in conjunction with the fact that  $R=\omega_{\omega_R}$ we can deduce that $\Ht^{4}_\mam(\omega_R)\neq 0$. Therefore the exact sequence (\ref{ExactSequence4}) yields $\ext^1_R(\omega_R,\omega_R)\neq 0$.
        \item[(iii)]  In spite of the argument of the previous part,  $\ext^1_R(\omega_R,\omega_R)=0$ whenever $R$ and the formal fibers of $R$ satisfy the $S_3$-condition.  Namely,  under this condition we can pass to the completion $\Rc$ of $R$ to use the exact sequence (\ref{ExactSequence4}). Since, $R=\omega_{\omega_R}$ so our hypothesis in conjunction with Theorem \ref{AttPrimeSnLocalCohomology}(ii) implies that $\Ht^{d-1}_\mam(\omega_R)=0$. Hence our claim follows from the exact sequence (\ref{ExactSequence4}).
        
        \item[(iv)]  Assume that $R$ and formal fibers of $R$ satisfies $S_4$, $\dep_R(\omega_R)\ge 3$  and $\ext^2_R(\omega_R,\omega_R)$ has finite length. We have $\ext^1_R(\omega_R,\omega_R)=0$ by the previous part. We claim, furthermore, that $\ext^2_R(\omega_R,\omega_R)=0$.  We assume, again, that $R$ is complete.   Since $R=\omega_{\omega_R}$ we can  use  Theorem \ref{AttPrimeSnLocalCohomology}(ii)   to conclude that $\Ht^{d-2}_\mam(\omega_R)=\Ht^{d-1}_\mam(\omega_R)=0$.  Therefore the exact sequence (\ref{ExactSequence4}) yields $$\ext^2_R(\omega_R,\omega_R)^\vee\cong \Ht^{d-1}_\mam(R)\tenr \omega_R.$$ In addition by Reminder \ref{PreliminariesCanonical}(iv) and \cite[Lemma 1.1]{Aoyama2} we have, $\Supp(\omega_R)=\ts(R)$. So,
        \begin{align*} 
        \At_R\big(\Ht^{d-1}_\mam(R)\tenr \omega_R\big)&=\As_R\Big(\big(\Ht^{d-1}_\mam(R)\tenr \omega_R\big)^\vee\Big)=\As_R\Big(\homm_R\big(\omega_R,\Ht^{d-1}_\mam(R)^\vee\big)\Big)&\\ &=\As_R\big(\Ht^{d-1}_\mam(R)^\vee\big)\ins\Supp(\omega_R)
        =\At_R\big(\Ht^{d-1}_\mam(R)\big).
        \end{align*}
        Hence in view of  the above identity and Lemma \ref{AttachPrimeDepthCanonicalModuleLemma}(ii) we have, $\mam\notin \At_R\big(\Ht^{d-1}_\mam(R)\tenr \omega_R\big)$. Now the fact that $\As_R\big(\ext^2_R(\omega_R,\omega_R)\big)\subseteq \{\mam\}$ implies that $\ext^2_R(\omega_R,\omega_R)=0$. 
        
        In fact, for the vanishing of the second cohomology module $\ext^2_R(\omega_R,\omega_R)$ we imposed several additional assumptions. Perhaps, this  shows  that the vanishing of $\ext^2_R(\omega_R,\omega_R)$ is considerably more complicated  than the vanishing of $\ext^1_R(\omega_R,\omega_R)$.
      \end{enumerate} 
   \end{rem}

   \section{Several Characterizations of Quasi-Gorenstein Rings}
   
     \begin{defi} 
     Let $M$ be an $R$-module and $\mathbf{x}:=x_1,\ldots,x_v$ be a sequence of elements of
$R$. Following \cite{Huneke}, by limit closure of $\mathbf{x}$ on $M$ we mean, $\bigcup\limits_{t\ge0}\big((x_{1}^{t+1},\ldots,x_{v}^{t+1})M:_{M}x_{1}^{t}\ldots x_{v}^{t}\big)$ and we will denote it by $\{\mathbf{x}\}^{\lim}_{M}$.
    \end{defi}          
    
    \begin{rem}\label{StartOfAllRemarks} Let $M$ be an $R$-module and $\mathbf{x}:=x_1,\ldots,x_v$ be a sequence of elements of $R$. The following statements hold.
        \item[(i)] It is easily verified that the submodule, $\{\mathbf{x}^i\}^{\lim}_M=\bigcup\limits_{t\ge0}\big((x_{1}^{it+i},\ldots,x_{v}^{it+i})M:_{M}x_{1}^{it}\ldots x_{v}^{it}\big),$ of $M$ coincides with  $\bigcup\limits_{t\ge0}\big((x_{1}^{i+t},\ldots,x_{v}^{i+t})M:_{M}x_{1}^{t}\ldots x_{v}^{t}\big)$ for each $i\in \mathbb{N}$.
        \item [(ii)] For each $i,j\in \mathbb{N}$ with $i\le j$, we denote the the multiplication map,  $$x^{j-i}_1\ldots x^{j-i}_v:M/(\mathbf{x}^{i})M\rightarrow M/(\mathbf{x}^{j})M,$$ by $\varphi_{i,j}$.
 So we have the direct system $\big(M/(\mathbf{x}^{i})M,\varphi_{i,j}\big)_{i,j\in\mathbb{N}}$ whose direct limit $\lim\limits_{\underset{i\in \mn}{\longrightarrow}}\big(M/(\mathbf{x}^i)M\big)$ is the top local cohomology  $\Ht^v_{(x_1,\ldots,x_v)}(M)$ (c.f. \cite[5.2.9. Theorem]{LocalCohomology}).
       \item[(iii)] By part (i),  the  kernel of the natural $R$-homomorphism $\lambda_i:M/(\mathbf{x}^i)M\rightarrow \lim\limits_{\underset{l\in \mn}{\longrightarrow}}M/(\mathbf{x}^l)M$, i.e. the module, $\bigg(\bigcup\limits_{t\ge0}\big((x_{1}^{i+t},\ldots,x_{v}^{i+t})M:_{M}x_{1}^{t}\ldots x_{v}^{t}\big)\bigg)/(\mathbf{x}^i)M$, is  the submodule $\{\mathbf{x}^i\}^{\lim}_M/(\mathbf{x}^i)M$ of $M/(\mathbf{x}^i)M$. In particular, we have the induced $R$-monomorphism, $$\tilde{\lambda_i}:M/\{\mathbf{x}^i\}^{\lim}_M\rightarrow \lim\limits_{\underset{l\in \mn}{\longrightarrow}}\big(M/(\mathbf{x}^l)M\big)\cong \Ht^v_{(x_1,\ldots,x_v)}(M).$$
       
       \item[(iv)] It is easily verified that, similar to part (ii), we have a direct system $\big(M/\{\mathbf{x}^i\}^{\lim}_M,\psi_{i,j}\big)_{i,j\in \mn}$ for which the $R$-homomorphism $\psi_{i,j}:M/\{\mathbf{x}^i\}^{\lim}_M\rightarrow M/\{\mathbf{x}^j\}^{\lim}_M$ is, again, the multiplication map by $x_1^{j-i}\ldots x_v^{j-i}$. Furthermore, for each $i,j\in \mn$ the following diagram is commutative,
              \[
        \xymatrix{
             &M/\{\mathbf{x}^i\}^{\lim}_M \ar[r]^{\tilde{\lambda_i}} \ar[dr]_{\psi_{i,j}} & \lim\limits_{\underset{l\in \mn}{\longrightarrow}}\big(M/(\mathbf{x}^l)M\big)\\
             & & M/\{\mathbf{x}^j\}^{\lim}_M      \ar[u]_{\tilde{\lambda_j}}   \\
             }
     \]
     The injectivity of $\tilde{\lambda_i}$  together with above commutative diagram implies that $\psi_{i,j}$ is an $R$-monomorphism for each $i,j\in \mn$. In addition, we have the natural isomorphism, $$\lim\limits_{\underset{i\in \mn}{\longrightarrow}}M/\{\mathbf{x}^i\}^{\lim}_M\rightarrow\Ht^v_{(x_1,\ldots,x_v)}(M)\cong \lim\limits_{\underset{l\in \mn}{\longrightarrow}}M/(\mathbf{x}^l)M.$$

          \item[(v)]  By \cite[3.2 Theorem]{O'Carroll} $\{\mathbf{x}\}^{\lim}_M=\mathbf{x}M$ provided $\mathbf{x}$ is a regular sequence on $M$\footnote{It is worth pointing out that Remark \ref{StartOfAllRemarks}(v) remains true if we drop both of the assumptions for the ring $R$ to be Noetherian and local.}.
          \item[(vi)] (\cite[Proposition 2.3.]{Marley}) Suppose that $M$ is finitely generated and $\{\mathbf{x}\}^{\lim}_M=\mathbf{x}M$. Then $x_1,\ldots,x_v$ is a regular sequence on $M$.
         
    \end{rem}

    \begin{rem} \label{SocLocalCohomologyMuCanonicalModule} Let $M$ be an $R$-module. Then,
    
    \item[(i)] \cite[Definition 1.2.18]{Herzog} We use the notation $\text{Soc}(M)$ to denote the socle of $M$, i.e. the $R$-module $\homm_R(R/\mam,M)\cong 0:_M\mam$ which is isomorphic to the sum of the simple submodules of $M$.  Furthermore the length of $M$  as $R$-module is denoted by $\lambda_R(M)$. 
    
    \item[(ii)]  Assume that $R$ has a canonical module $\omega_R$. Then,
     \begin{center}
       $\bigg(\text{Soc}\big(H_{\mam}^{d}(R)\big)\bigg)^{\vee}\cong H_{\mam}^{d}(R)^{\vee}\bigotimes_{R}(R/\mam)\cong(\omega_{R}\bigotimes_{R}\widehat{R})\bigotimes_{R}(R/\mam)\cong(\omega_{R}/\mam\omega_{R})$.
     \end{center}
    \item[(iii)] In view of (ii) we have,
      \begin{align*}
      \mu(\omega_R)=\text{Vdim}_{R/\mam}(\omega_R/\mam\omega_R)&=
      \text{Vdim}_{R/\mam}\Bigg(\bigg(\text{Soc}\big(\Ht^d_\mam(R)\big)\bigg)^\vee\Bigg)
       =\lambda_R\Bigg(\bigg(\text{Soc}\big(\Ht^d_\mam(R)\big)\bigg)^\vee\Bigg)&\\ &=
      \lambda_R\bigg(\text{Soc}\big(\Ht^d_\mam(R)\big)\bigg) =
      \text{Vdim}_{R/\mam}\bigg(\text{Soc}\big(\Ht^d_\mam(R)\big)\bigg).
      \end{align*}
   Consequently,
      \begin{center}
        $\mu(\omega_{\widehat{R}})=
        \text{Vdim}_{\widehat{R}/\mam\widehat{R}}
        \bigg(\homm_{\widehat{R}}\big(\widehat{R}/\mam\widehat{R},\Ht^d_{\mam\widehat{R}}(\widehat{R})\big)\bigg)=\text{Vdim}_{R/\mam}\bigg(\homm_{R}\big(R/\mam,\Ht^d_\mam(R)\big)\bigg)$.
      \end{center}
    \end{rem}

    \begin{lem} \label{InjectiveDimensionAndCompletion} Let $A$ be an Artinian $R$-module. It is well-know that since $A$ is an $\mam$-torsion $R$-module so $A$ has a natural $\widehat{R}$-module structure (see  \cite[10.2.9 Remark.]{LocalCohomology}). We have $\text{injdim}_R(A)=\text{injdim}_{\widehat{R}}(A)$.
      \begin{proof}
        Let $\mathcal{E}^\bullet$ be a minimal injective resolution of $A$. Recall that, as $A$ is an Artinian module,  $\mathcal{E}^i$ is a finite copy of $\text{E}_R(R/\mam)$ for each $i\ge 0$ and that we can endow both of $A$ and $\text{E}(R/\mam)$ with an $\widehat{R}$-module structure such that  $\text{E}_R(R/\mam)\cong \text{E}_{\widehat{R}}(\widehat{R}/\mam\widehat{R})$ and $\mathcal{E}^\bullet$ is a minimal $\widehat{R}$-injective resolution of $A$ too. Thus $\text{injdim}_R(A)= \text{injdim}_{\widehat{R}}(A)$.
      \end{proof}
    \end{lem}
    
    \begin{defi} \cite{BeyondTotallyReflexive} An $R$-module $N$ is called Gorenstein injective (Gorenstein flat)  if there exists an exact complex $\mathbf{I}$ of injective $R$-modules (exact complex $\mathbf{F}$ of flat $R$-modules) such that $\text{Ker}(I^0 \rightarrow I^1)]\cong N$ ($\text{Coker}(F_1\rightarrow F_0)\cong N$) and such that $\homm_R(E,\mathbf{I})$ ($E\tenr \mathbf{F}$) is exact for every injective $R$-module $E$.
    \end{defi}
    
    \begin{defi} \cite{BeyondTotallyReflexive} An (augmented) Gorenstein injective (respectively, Gorenstein flat) resolution of $M$ is an exact sequence, $0\rightarrow M\rightarrow B^0\rightarrow B^1\rightarrow\cdots\rightarrow B^i\rightarrow\cdots$ (respectively, $\cdots\rightarrow F_i\rightarrow \cdots\rightarrow F_1\rightarrow F_0\rightarrow M\rightarrow 0$), where each $B^i$ (respectively, each $F_i$) is Gorenstein injective (respectively, Gorenstein flat). The Gorenstein injective (respectively, Gorenstein flat) dimension of $N$, denoted by $\text{Ginjdim}_RN$ (respectively, $\text{Gfdim}_RN$), is the least integer  $n\ge 0$ such that there exists a Gorenstein injective resolution of $N$ (respectively, Gorenstein flat resolution of $N$) such that $B^i=0$ (respectively, $F_i=0$) for each $i>n$ (If there does not exist such a resolution then we say that $\text{Ginjdim}_RN=\infty$ (respectively, $\text{Gfdim}_RN=\infty$) ).
    \end{defi}
    
      According to the  parts  (v) and (vi) of Remark \ref{StartOfAllRemarks} we can deduce that $R$ is Cohen-Macaulay if and only if the assertion $\{\mathbf{x}\}^{\lim}_R=\mathbf{x}R$ holds for some (equivalently for every) system of parameters $\mathbf{x}$ of $R$. Also recall that a Cohen-Macaulay ring is unmixed. Therefore the next theorem is a generalization of the fact that a Cohen-Macaulay ring  is Gorenstein if and only if each (some)
parameter ideal is irreducible.
    
     \label{Socdim_Local_Cohomology_Miu_Cananoical}

     \begin{thm} \label{characterization} The following statements are equivalent.
       \begin{enumerate}
         \item[(i)] $\text{injdim}\big(H^d_\mam(R)\big)<\infty$.
         \item[(ii)] $\text{Ginjdim}_R\big(\Ht^d_\mam(R)\big)<\infty$ and $\ext^i_{\Rc}(\omega_{\widehat{R}},\omega_{\widehat{R}})=0$, for each $1\le i\le \dep_{\Rc}(\omega_{\Rc})$.        
         \item[(iii)] $R$ is quasi-Gorenstein. 
         \item[(iv)] $\widehat{R}$ is unmixed and the top local          
cohomology  module of $R$ with respect to $\mathfrak{m}$ has one dimensional socle.         
         \item [(v)] $\widehat{R}$ is unmixed and $\{\mathbf{x}\}^{\lim}_R$ is irreducible
for every system of parameters $\mathbf{x}:= x_1,\ldots,x_d$ of $R$.
         \item [(vi)]  $\widehat{R}$ is unmixed and there exists a system of parameter $\mathbf{x}:= x_1,\ldots,x_d$ of $R$ such that $\{\mathbf{x}^n\}^{\lim}_R$ is irreducible for each $n\in \mathbb{N}$.
          \item[(vii)] $\widehat{R}$ is unmixed and $\{\mathbf{x}\}^{\lim}_R$ is an irreducible ideal for some system of parameters $\mathbf{x}$ contained in  a high enough power of $\mam$.

       \end{enumerate}
       \begin{proof}
       (i)$\Leftrightarrow$(iii): One implication is clear. For the reverse, recall that $\Ht^d_\mam(R)$  is an Artinian $R$-module and that $\Ht^d_\mam(R)\cong \Ht^d_{\mam\widehat{R}}(\widehat{R})$ with its induced $\widehat{R}$-module structure. So by Lemma \ref{InjectiveDimensionAndCompletion} we can assume that $R$ is  complete and whence $R$ has a canonical module $\omega_R$. But the finiteness of $\text{injdim}\big(\Ht^d_\mam(R)\big)$ implies that $\text{fd}_R(\omega_R)<\infty$. Therefore by \cite[Theorem 7.10.]{Matsumura} $\pd_R(\omega_R)=\text{fd}_R(\omega_R)<\infty$. Hence  $\omega_R$ is free by virtue of \cite[Theorem 3]{Aoyama3}. 
       
       (ii)$\Leftrightarrow$(iii): If $\text{Ginjdim}_R\big(\Ht^d_\mam(R)\big)<\infty$ then  \cite[Lemma  3.5.]{Sazeedeh} implies that $\text{Ginjdim}_{\Rc}\big(\Ht^d_{\mam\Rc}(\Rc)\big)<\infty$. Hence using
        \cite[Theorem 4.16.]{BeyondTotallyReflexive} and \cite[Proposition 4.24]{BeyondTotallyReflexive} we conclude that $\text{Gdim}_{\Rc}(\omega_{\Rc})<\infty$. Thus the statement follows from Theorem \ref{G-dim}. The reverse is  straightforward.
       
       (iii)$\Rightarrow$(iv): $\Rc$ is quasi-Gorenstein. So by Reminder \ref{PreliminariesCanonical}(iv) $\Rc$ is an unmixed ring. Since $\homm_R\big(R/\mam,\Ht^d_\mam(R)\big)\cong \homm_R\big(R/\mam,\text{E}(R/\mam)\big)\cong R/\mam$, we have $\text{Vdim}_{R/\mam}\Big(\text{Soc}\big(\Ht^d_\mam(R)\big)\Big)=1.$
       
       (iv)$\Rightarrow$(v): By Remark \ref{StartOfAllRemarks}(iii) there is an $R$-monomorphism $R/\{\mathbf{x}^i\}_{R}^{\lim}\rightarrow\lim\limits_{\underset{n\in\mathbb{N}}{\longrightarrow}}R/(\mathbf{x}^{n})\cong H_{\mathfrak{m}}^{d}(R)$  for each $i\in \mathbb{N}$. So  the fact that $\text{Vdim}_{R/\mam}\Big(\text{Soc}\big(H^d_\mathfrak{m}(R)\big)\Big)=1$ implies that either $R/\{\mathbf{x}^i\}^{\lim}_R$ is  zero or it has one dimensional socle. Let, $\{\mathbf{x}^i\}^{\lim}_R\neq R$. Then since $R/\{\mathbf{x}^i\}^{\lim}_R$ is an Artinian ring with $\text{Vdim}_{R/\mam}\big(\homm_R(R/\mam, R/\{\mathbf{x}^i\}^{\lim}_R)\big)=1$, by \cite[Theorem 18.1.]{Matsumura} $R/\{\mathbf{x}^i\}^{\lim}_R$ is a Gorenstein Artinian ring and consequently the zero ideal of $R/\{\mathbf{x}^i\}^{\lim}_R$ is irreducible, i.e. $\{\mathbf{x}^i\}^{\lim}_R$ is an irreducible ideal of $R$. Also if, $\{\mathbf{x}^i\}^{\lim}_R=R$\footnote{Here, it is worth to mention the following. The Hochster's Monomial Conjecture states that for each local ring $R$ and each system of parameters $\mathbf{y}$ of $R$ we have $\{\mathbf{y}\}^{\lim}_R\subsetneq R$. Hence in the case where either both of $R$ and $R/\mam$ have the same characteristic or $\dime(R)\le 3$, we already know that $\{\mathbf{x}\}^{\lim}_R\subsetneq R$. }, then $\{\mathbf{x}^i\}^{\lim}_R$ is an irreducible ideal of $R$.

       (v)$\Rightarrow$(vi): It is obvious.

       (vi)$\Rightarrow$(iii):       By our hypothesis, for sufficiently large $n$, $R/\{\mathbf{x}^n\}^{\lim}_R$ is a Gorenstein Artinian ring and consequently $R/\{\mathbf{x}^n\}^{\lim}_R$ has one-dimensional socle. On the other hand  $\text{Soc}\big(H^d_\mathfrak{m}(R)\big)$ is a  non-zero finitely generated $R$-module  and, $$\text{Soc}\big(H_{\mathfrak{m}}^{d}(R)\big)= \text{Soc}(\lim\limits_{\underset{n\in\mathbb{N}}{\longrightarrow}}R/\{\mathbf{x}^{n}\}_{R}^{\lim})\cong\lim\limits _{\underset{n\in\mathbb{N}}{\longrightarrow}}\text{Soc}(R/\{\mathbf{x}^{n}\}_{R}^{\lim}).$$ It turns out that,
           $\text{Soc}(R/\{\mathbf{x}^{n}\}_{R}^{\lim})\rightarrow \text{Soc}(H_{\mathfrak{m}}^{d}(R))$   is onto for some $m\in \mathbb{N}$ and each $n\ge m$ and thereby $\text{Vdim}_{R/\mam}\big(\text{H}^{d}_\mam(R)\big)=1$. Hence  $\omega_{\widehat{R}}$ is cyclic by Remark \ref{SocLocalCohomologyMuCanonicalModule}(iii). On the other hand our unmixedness hypothesis together with the identity (\ref{AnnihilatorCanonicalModule}) of Reminder \ref{PreliminariesCanonical} yields, $0:_{\Rc}\omega_{\Rc}=0$, i.e. $\omega_{\Rc}\cong \Rc$.

              (vii)$\Rightarrow$(iv):  Let  $\ell_d(R)\in \mn$ be as in \cite[Definition 2.4.]{Marley}.  Suppose that $\mathbf{x}$ is a system of parameters in $\mam^{\ell_d(R)}$ such that $\{\mathbf{x}\}^{\lim}_R$ is an irreducible ideal. Precisely as in \cite[Proposition 2.5.]{Marley},    \cite[Lemma 3.12]{Goto}     implies that the natural map $\text{Soc}\big(R/(\mathbf{x})\big)\rightarrow \text{Soc}\big(\lim\limits_{\underset{n\in \mathbb{N}}{\longrightarrow}}R/(\mathbf{x}^n)\big)$ is onto. By Remark \ref{StartOfAllRemarks}(iv),  there exists a natural isomorphism $       \lim\limits_{\underset{n\in \mathbb{N}}{\longrightarrow}}R/\{\mathbf{x}^n\}^{\lim}_R \rightarrow \lim\limits_{\underset{n\in \mathbb{N}}{\longrightarrow}}R/(\mathbf{x}^n)$ which yields the following commutative diagram,
            \begin{center}
            $\begin{CD}
              \text{Soc}\big(R/(\mathbf{x})\big) @>>> \text{Soc}\big(\lim\limits_{\underset{n\in \mathbb{N}}{\longrightarrow}}R/(\mathbf{x}^n)\big)\\
              @VVV @AA\cong A\\
              \text{Soc}(R/\{\mathbf{x}\}^{\lim}_R) @>>> \text{Soc}(\lim\limits_{\underset{n\in \mathbb{N}}{\longrightarrow}}R/\{\mathbf{x}^n\}^{\lim}_R).
            \end{CD}$
            \end{center}
            This implies that $\text{Soc}(R/\{\mathbf{x}\}^{\lim}_R)\rightarrow \text{Soc}\big(\lim\limits_{\underset{n\in \mathbb{N}}{\longrightarrow}}R/\{\mathbf{x}^n\}^{\lim}_R\big)\cong \text{Soc}\big(H^d_\mam(R)\big)$ is an epimorphism too and whence $\text{Socdim}\big(H^d_\mam(R)\big)=1$. 
            
            (v)$\Rightarrow$(vii): It is obvious.

       \end{proof}
    \end{thm}

     It is worth pointing out that the equivalence (iv)$\Leftrightarrow$(iii) in the preceding theorem may be considered as a counterpart to the well-known fact that $R$ is a Gorenstein ring if and only if $R$ is a  Cohen-Macaulay ring of type 1. Also the equivalence (vii)$\Leftrightarrow$(iii) is a generalization of the main result of \cite[Theorem 2.7.]{Marley}, i.e. the following: \footnote{This is mentioned in the introduction of \cite{Marley}}.

   \paragraph{Theorem: (\cite{Marley}}) There exists an integer $\ell$ such that $R$ is Gorenstein if and only if some parameter ideal contained in $\mathfrak{m}^\ell$ is irreducible.     
   
   \begin{rem}  Assume that $R$ has a canonical module $\omega_R$. By the identity (\ref{AnnihilatorCanonicalModule}) of Reminder \ref{PreliminariesCanonical}, $R$ is unmixed if and only if $0:_R\omega_R=0$ if and only if $0:_{\widehat{R}}(\omega_R\bigotimes_R \widehat{R})=0$ if and only if $\widehat{R}$ is unmixed. Hence we can replace $\widehat{R}$ with $R$  in  all parts of the Theorem \ref{characterization} provided that $R$ has a canonical module. However,  \cite[EXAMPLE 2.3]{Nishimura} gives us an  example of  a Noetherian local domain such that its completion  is not unmixed.
   \end{rem}  
   
   In the following  example we construct a ring $R$ such that $H^d_\mam(R)$    has a one dimensional socle but it is not quasi-Gorenstein.  This implies that  in Theorem \ref{characterization} the unmixedness condition of $\widehat{R}$ is necessary.
   
   \begin{exam} 
   Let $K$ be a field and consider the power series ring $K[[X,Y]]$. Then, $$R=K[[X,Y]]/(X^5Y^5),$$ is a Gorenstein ring. Let  $\maa=(X^4Y^3,X^3Y^4)/(X^5Y^5)$ be an ideal of $R$.  An easy verification shows that $0:_R\maa$ is a principal ideal generated by the image of $X^2Y^2$ in $R$ and that $\maa\subsetneq 0:_R(0:_R\maa)=(X^3Y^3)/(X^5Y^5)$. Note that $\hit_R(\maa)=0$. This implies that $0:_R\maa\cong \homm_R(R/\maa,R)$ is the canonical module of $R/\maa$ which is cyclic but $0:_{R/\maa}(0:_R\maa)=0:_R(0:_R\maa)/\maa\neq 0$. Therefore the canonical module of $R/\maa$ is not free and $R/\maa$ is not quasi-Gorenstein. However by dualizing an epimorphism $R\rightarrow \omega_R$ we get an embedding $\Ht^d_\mam(R)\rightarrow \Et(R/\mam)$. Therefore, $\text{Vdim}_{\Rc/\mam\Rc}\Big(\text{Soc}\big(\Ht^d_\mam(R)\big)\Big)=1,$ although $R$ is not quasi-Gorenstein.
   \end{exam}
   
   It is well-known  that  a Cohen-Macaulay ring $R$ is Gorenstein if and only if there exists an irreducible system of parameters for $R$. So, in accordance with  Theorem \ref{characterization}, perhaps it is natural to ask whether $R$ is quasi-Gorenstein provided $\widehat{R}$ is unmixed and there exist a system of parameters for $R$ whose limit closure is irreducible? The answer is negative.

      \begin{exam} Let $R=\mathbb{C}[[X,Y,Z,T]]/(XY,XT,ZY,ZT)$. We denote  the images of $X,Y,Z,T$ in $R$ by $x,y,z,t$, respectively. Then $R$ is a non-Cohen-Macaulay unmixed complete local ring of dimension $2$. One can verify that $x+y,z+t$ is a system of parameters for $R$ and that, $$\big((x+y)^2,(z+t)^2\big):(x+y)(z+t)=(x,y,z,t).$$ So $(x+y,z+t)^{\text{lim}}$  is the unique maximal ideal of $R$ and in particular it is irreducible. But $R$ is not quasi-Gorenstein as $R$ is a non-Cohen-Macaulay ring of dimension $2$.
      \end{exam}

  In the Theorem \ref{characterization} the quasi-Gorensteinness  is characterized with the aid of the limit closure of parameter ideals. We would like to digress momentarily to give another such application of the limit closure. The following proposition will be use in the proof of Theorem \ref{GCMThm}.
  
  \begin{prop}\label{GCMProp}
    Suppose that there exists $n\in\mn$ such that $\mam^n(\{\mathbf{x}\}^{\lim}_M/\mathbf{x}M)=0,$  for each s.o.p. $\mathbf{x}$ for $M$. Then, $\mam^n\Big(\big((x_1,\ldots,x_i)M:_{M}x_{i+1}\big)/(x_1,\ldots,x_i)M\Big)=0$, for each s.o.p. $\mathbf{x}$ for $M$ and for every $0\le i\le d'-1$.
      \begin{proof}
         Let, $m\in \big((x_1,\ldots,x_i)M:_{M}x_{i+1}\big)$, where $0\le i\le d'-1$. Then, $$x_{i+1}m\in (x_1,\ldots,x_i,x_{i+1}^{u+1},x_{i+2}^{u},\ldots,x_{d'}^{u})M,$$ for each $u\in \mn$. Therefore, $x_1\ldots x_{i}x_{i+1}^u\ldots x_{d'}^um\in (x_1^{2},\ldots,x_i^{2},x_{i+1}^{2u},\ldots,x_{d'}^{2u})M$  which yields $m\in \{x_1,\ldots,x_i,x_{i+1}^u,\ldots,x_{d'}^u\}^{\lim}_M$. Thus our hypothesis yields $\mam^nm\in (x_1,\ldots,x_i,x_{i+1}^u,\ldots,x_{d'}^u)M$ for each $u\in\mn$. Now the assertion follows from the Krull's intersection theorem.
      \end{proof}
  \end{prop}

  \begin{thm} \label{GCMThm} The following statements hold.
    \begin{enumerate}
      \item[(i)] If $R$ is a generalized Cohen-Macaulay  ring, then there exists $n\in \mn_0$ such that, $\mam^n(\{\mathbf{x}\}^{\lim}_R/\mathbf{x}R)=0$, for each s.o.p. $\mathbf{x}$ for $R$\footnote{See  \cite[9.5.7 Exercise]{LocalCohomology} for the definition of Generalized Cohen-Macaulay rings.}. The reverse also holds whenever $R$ is a homomorphic image of a Cohen-Macaulay ring.
      \item[(ii)] $R$ is a Buchsbaum\footnote{See \cite[Page 14, Theorem 2 and definition]{Stuckrad} for the definition of Buchsbaum rings.} ring if and only if   $\mam(\{\mathbf{x}\}^{\lim}_R/\mathbf{x}R)=0$  for each s.o.p. $\mathbf{x}$ for $R$.
    \end{enumerate}
    \begin{proof} If there exists $n\in\mathbb{N}$ such that   $\mam^n\big(\{\mathbf{x}\}^{\lim}_R/\mathbf{x}R\big)=0$  for each s.o.p.  $\mathbf{x}$ for $R$ then by Proposition \ref{GCMProp}, $\mam^n\big((x_1,\ldots,x_i)R:_Rx_{i+1}\big)\subseteq (x_1,\ldots,x_i)R$, for each s.o.p $\mathbf{x}$ of $R$ and $0\le i\le d-1$. Thus $R$ is generalized Cohen-Macaulay by \cite[Page 260, Proposition 16]{Stuckrad}. By the same token  each s.o.p. for $R$ is a weak $\mam$-sequence of $R$, i.e. $R$ is a Buchsbaum ring, provided $n\le 1$.

       Conversely assume that $R$ is generalized Cohen-Macaulay. Let, $\mam^{\alpha_i}\Ht^i_\mam(R)=0$, for each $0\le i\le d-1$. Choose an integer $n\in \mathbb{N}_0$ satisfying, $n\ge \sum\limits_{i=0}^{d-1}\alpha_i$.
            We show that, $\mam^{(2^{d}-1)n}\{\mathbf{x}\}^{\lim}_R\subseteq \mathbf{x}R$, for each s.o.p. $\mathbf{x}$ for $R$. To this aim we induct on $d\ge 1$ (for the case where $d=0$ the assertion is obvious). If, $d=1$, then for every s.o.p. $x$ of $R$  there exists $t\in \mathbb{N}$ such that $\{x\}^{\lim}_R=x^{t+1}R:_Rx^t=(0:_Rx^t)+xR$. So by \cite[Corollary 8.1.3.(b)]{Herzog} we are done in this case. Suppose that the statement is true for smaller values of $d$. Let $\mathbf{x}$ be a s.o.p. for $R$ and  $r\in \{\mathbf{x}\}^{\lim}_R$, i.e. $x_1^j\ldots x_{d}^jr\in (x_1^{j+1},\ldots,x_d^{j+1})$ for some $j\in \mathbb{N}$. Then there exists $r^\prime\in R$ such that, $$x_1^j\ldots x_{d-1}^jr-x_dr^\prime \in (x_1^{j+1},\ldots,x_{d-1}^{j+1})R:_{R}x_d^j.$$ Thus,   \cite[Corollary 8.1.3.(b)]{Herzog} shows that,  $$\mam^n(x_1^j\ldots x_{d-1}^jr)\subseteq (x_1^{j+1},\ldots,x_{d-1}^{j+1},x_{d})R.$$ Therefore, $\overline{\mam}^n\overline{r}\subseteq \{\overline{x_1},\ldots,\overline{x_{d-1}}\}^{\lim}_{\overline{R}}$, where  the notation $\overline{\ }$ means modulo $x_{d}R$. We claim that, $2n\ge \sum\limits_{i=0}^{d-2}\beta_i$, where $\beta_i$ is the least integer satisfying $\overline{\mam}^{\beta_i}\subseteq 0:_{\overline{R}}\big(\Ht^i_{\overline{m}}(\overline{R})\big)$. But we defer the proof of the claim for a while to see what  it implies. Using our claim in conjunction with the inductive hypothesis we can deduce that $\overline{\mam}^{(2^{{d}-1}-1)2n}\overline{\mam}^n\overline{r}\subseteq (\overline{x}_1,\ldots,\overline{x}_{{d}-1})\overline{R}$. So, $\mam^{(2^{d}-1)n}r\subseteq (x_1,\ldots,x_{d})R$,  as desired.    
            
            Now, we prove our claim. By \cite[Corollary 8.1.3.]{Herzog} we have, $0:_Rx_d\subseteq \Gamma_{\mam}(R)$. Hence the exact sequence, $0\rightarrow 0:_Rx_d\rightarrow R\rightarrow R/(0:_Rx_d)\rightarrow 0,$ shows that $\Ht^i_{\mam}(R)\cong \Ht^i_{\mam}(R/0:_{R}x_d)$ for each $i\ge 1$. Therefore using the exact sequence, $0\rightarrow R/(0:_Rx_d)\overset{x_d}{\rightarrow}R\rightarrow \overline{R}\rightarrow 0,$ we can deduce that, $\overline{\mam}^{\alpha_i+\alpha_{i+1}}\Ht^i_{\overline{\mam}}(\overline{R})=0$ for each $0\le i\le d-2$, provided $\mam^{\alpha_i}\Ht^i_\mam(R)=0$ for each $0\le i\le d-1$. Since, $\sum\limits_{i=0}^{d-1}\alpha_i\le n$. so, $\sum\limits_{i=0}^{d-2}(\alpha_i+\alpha_{i+1})\le 2n$. This proves our claim.
       
       Furthermore if $R$ is a Buchsbaum ring then by \cite[Proposition 1.17]{Stuckrad} every s.o.p. for $R$ is both of an unconditioned strong $d$-sequence\footnote{See  \cite[Definition 1.2.]{FrobeniusTestExponent} for the definition of an unconditional strongly $d$-sequence.} and also a weak $\mam$-sequence. Hence by virtue of \cite[Theorem 3.6.(ii)]{FrobeniusTestExponent} we have, $\{\mathbf{x}\}^{\lim}_R=\sum\limits_{j=1}^d\big((x_1,\ldots,\widehat{x_j},\ldots,x_d):x_j\big)$. Since $\mathbf{x}$ is a weak $\mam$-sequence, so we have, $\mam\big((x_1,\ldots,\widehat{x_i},\ldots,x_d):x_i\big)\subseteq (x_1,\ldots,\widehat{x_i},\ldots,x_d)$. This, immediately gives us, $\mam\{\mathbf{x}\}^{\lim}_R\subseteq \mathbf{x}R$.
    \end{proof}
  \end{thm}
      \section*{Acknowledgement}
      We would like to express our deepest gratitude to  Raheleh Jafari for many fruitful discussions and helpful suggestions, most notably for Remark \ref{Jafari}. We are also grateful to Kamran Divaani-Aazar for  valuable comments.
      
    \begin{center}\section*{References}\end{center}

   \small \textsc{Ehsan Tavnafar, Department of Mathematics, Shahid Beheshti University, G.C., Tehran, Iran.}  \\
E-mail address: \href{mailto:tavanfar@ipm.ir}{tavanfar@ipm.ir} \\ 

\small \textsc{M. Tousi, Department of Mathematics, Shahid Beheshti University, G.C., Tehran, Iran.} \newline
E-mail address: \href{mailto:mtousi@ipm.ir}{mtousi@ipm.ir}


\begin{biblist}
       \bib{Aoyama3}{article}{
         author={Y. Aoyama},
         title={On the depth and the projective dimension of the canonical module},
         journal={Japan. J. Math. (N.S.),},
         volume={6,},
         number={1},
         pages={61-66},
         year={1980},
       }
       \bib{Aoyama}{article}{
         author={Y. Aoyama},
         title={Some basic results on canonical modules},
         journal={J. Math. Kyoto Univ.,},
         volume={23,},
         number={1},
         pages={85-94},
         year={1983},
       }
       \bib{Aoyama2}{article}{
         author={Y. Aoyama and S. Goto},
         title={On the endomorphism ring of the canonical module},
         journal={J. Math. Kyoto Univ.,},
         volume={25-1,},
         pages={21-30},
         year={1985},
       } 

    \bib{Bass}{article}{
      Author = {H. {Bass}},
    Title = {{On the ubiquity of Gorenstein rings}},
    FJournal = {{Mathematische Zeitschrift}},
    Journal = {{Math. Z.,}},
    ISSN = {0025-5874; 1432-1823/e},
    Volume = {82,},
    Pages = {8--28},
    Year = {1963},
    }
    \bib {BertinAnneaux}{article}{
    	AUTHOR = {M. -J. Bertin},
    	TITLE = {Anneaux d'invariants d'anneaux de polynomes, en
    		caract\'eristique {$p$}},
    	JOURNAL = {C. R. Acad. Sci. Paris S\'er. A-B},
    	VOLUME = {264},
    	YEAR = {1967},
    	PAGES = {A653--A656},
    }
       \bib{LocalCohomology}{book}{
         author={M. P. Brodmann and  R. Y. Sharp},
         title={Local Cohomology: An Algebraic Introduction with Geometric Applications, second edition}
         publisher={Cambridge University Press}
         year={2013}
       }
       \bib{Herzog}{book}{
         author={W. Bruns and J. Herzog},
         title={Cohen-Macaulay rings},
         publisher={Cambridge University Press},
         Year={1993},
       }  
   \bib{Christensen}{book}{
      Author = {L. {W. Christensen}},
    Title = {{Gorenstein dimensions.}},
    ISBN = {3-540-41132-1/pbk},
    Pages = {viii + 204},
    Year = {2000},
    Publisher = {Berlin: Springer},
     }
       \bib{BeyondTotallyReflexive}{book}{
         Author = {L. W. {Christensen} and H-B\o rn {Foxby} and H. {Holm}},
    Title = {{Beyond totally reflexive modules and back. A survey on Gorenstein dimensions.}},
    BookTitle = {{Commutative algebra. Noetherian and non-Noetherian perspectives}},
    ISBN = {978-1-4419-6989-7/hbk; 978-1-4419-6990-3/ebook},
    Pages = {101--143},
    Year = {2011},
    Publisher = {New York, NY: Springer},
       }

    \bib{Dibaei}{article}{
         Author = {M. T. {Dibaei}},
    Title = {{A study of Cousin complexes through the dualizing complexes}},
    FJournal = {{Communications in Algebra}},
    Journal = {{Commun. Algebra,}},
    ISSN = {0092-7872; 1532-4125/e},
    Volume = {33,},
    Number = {1},
    Pages = {119--132},
    Year = {2005},
       }
       \bib{Jafari}{article}{
         Author = {M. T. {Dibaei} and R. {Jafari}},
    Title = {{Cohen-Macaulay loci of modules}},
    FJournal = {{Communications in Algebra,}},
    Journal = {{Commun. Algebra,}},
    ISSN = {0092-7872; 1532-4125/e},
    Volume = {39,},
    Number = {10},
    Pages = {3681--3697},
    Year = {2011},
       }
         \bib{Enochs}{article}{
         Author = {E. E. {Enochs} and O. M.G. {Jenda}},
    Title = {{Gorenstein injective and projective modules}},
    FJournal = {{Mathematische Zeitschrift}},
    Journal = {{Math. Z.}},
    ISSN = {0025-5874; 1432-1823/e},
    Volume = {220,},
    Number = {4},
    Pages = {611--633},
    Year = {1995},
       }
       \bib{Foxby}{article}{
          Author = {R. {Fossum,}  H-B {Foxby,}  P. {Griffith} and I. {Reiten}},
    Title = {{Minimal injective resolutions with applications to dualizing modules and Gorenstein modules,}},
    FJournal = {{Publications Math\'ematiques }},
    Journal = {{Publ. Math., Inst. Hautes \'Etud. Sci.,}},
    ISSN = {0073-8301; 1618-1913/e},
    Volume = {45,},
    Pages = {193--215},
    Year = {1975},
       }
     \bib{Fouli}{article}{
         author={L. Fouli and C. Huneke},
         title={What is a system of parameter?},
         journal={Proc. Am. Math. Soc.,},
         volume={139,},
         year={2011},
         pages={ 2681-2696},
       }
       \bib{Foxby2}{article}{
         Author = {H-B {Foxby}},
    Title = {{Isomorphisms between complexes with applications to the homological theory of modules}},
    FJournal = {{Mathematica Scandinavica}},
    Journal = {{Math. Scand.,}},
    ISSN = {0025-5521; 1903-1807/e},
    Volume = {40,},
    Pages = {5--19},
    Year = {1977},
       }
       
       
       
       
       \bib{GotoBuchsbaum}{article}{
         
         Author = {S. {Goto}},
    Title = {{A note on quasi-Buchsbaum rings,}},
    FJournal = {{Proc.  Am. Math. Soc.}},
    Journal = {{Proc. Am. Math. Soc.,}},
    Volume = {90,},
    Pages = {511--516},
    Year = {1984},
       }
   \bib {GotoOnTheAssociated}{article}{
   	AUTHOR = {S. Goto},
   	TITLE = {On the associated graded rings of parameter ideals in
   		{B}uchsbaum rings},
   	JOURNAL = {J. Algebra},
   	FJOURNAL = {Journal of Algebra},
   	VOLUME = {85},
   	YEAR = {1983},
   	NUMBER = {2},
   	PAGES = {490--534},
   }    
       \bib{Goto}{article}{
        author={S. Goto and H. Sakurai,}
        title={The equality $I^2=QI$ in Buchsbaum rings,}
        journal={Rend. Sem. Mat. Univ. Padova,}
        volume={110,},
        pages={25-56},
        year={2003}
       }
    \bib{GotoSegre}{article}{
      AUTHOR = {S. {Goto}  and  K. {Watanabe},},
      TITLE = {On graded rings. {I}},
      JOURNAL = {J. Math. Soc. Japan,},
      FJOURNAL = {Journal of the Mathematical Society of Japan},
      VOLUME = {30,},
      YEAR = {1978},
      NUMBER = {2},
      PAGES = {179--213},
    }   
    \bib{Hartshorne}{book}{
          Author = {R. {Hartshorne}},
    Title = {{Residues and duality. Appendix: Cohomologie \`a support propre et construction du foncteur $f\sp!$. par P. Deligne.}},
    Year = {1966},
    HowPublished = {{Lecture Notes in Mathematics. 20. Berlin-Heidelberg-New York: Springer-Verlag, 423 p. (1966).}},
       }
     \bib{Hassanzadeh}{article}{
         Author = {H. {Hassanzadeh Hafshejani,}  N. {Shirmohammadi} and H. {Zakeri}},
    Title = {{A note on quasi-Gorenstein rings}},
    FJournal = {{Archiv der Mathematik}},
    Journal = {{Arch. Math.,}},
    ISSN = {0003-889X; 1420-8938/e},
    Volume = {91,},
    Number = {4},
    Pages = {318--322},
    Year = {2008},

       }
         \bib{HeiznerUrlich}{article}{
      Author = {W. {Heinzer} and M. {Kim} and B. {Ulrich}},
      Title = {{The Cohen-Macaulay and Gorenstein properties of rings associated to filtrations}},
      FJournal = {{Communications in Algebra}},
      Journal = {{Commun. Algebra,}},
      ISSN = {0092-7872; 1532-4125/e},
      Volume = {39,},
      Number = {10},
      Pages = {3547--3580},
      Year = {2011},
    }  
       \bib{Hermann}{article}{
        author={M. Herrmann and N. Trung}
        title={Examples of Buchsbaum quasi-Gorenstein rings}
        journal={Proc.  Am. Math.  Soc.,}
        volume={117,},
        pages={619-625},
        year={1993}
      }
      \bib{CanonicalElements}{article}{
        Author = {M. {Hochster}},
        Title = {{Canonical elements in local cohomology modules and the direct summand conjecture}},
      FJournal = {{Journal of Algebra}},
    Journal = {{J. Algebra,}},
    ISSN = {0021-8693},
    Volume = {84,},
    Pages = {503--553},
    Year = {1983},
      }

      \bib{Huneke}{book}{
       author={C. Huneke},
       title={Tight closure, parameter ideals, and geometry, in Six Lectures on Commutative Algebra
(J. Elias, J.M. Giral, R.M. Mir´o-Roig, and S. Zarzuela, eds), Progress in Mathematics, vol. 166,
Birkhauser Verlag, Basel, 1998, 187-239.}       
      }

      \bib{FrobeniusTestExponent}{article}{
        Author = {C. {Huneke,}  M. {Katzman,}  R. Y. {Sharp} and Y. {Yao}},
    Title = {{Frobenius test exponents for parameter ideals in generalized Cohen-Macaulay local rings}},
    FJournal = {{Journal of Algebra}},
    Journal = {{J. Algebra,}},
    ISSN = {0021-8693},
    Volume = {305,},
    Number = {1},
    Pages = {516--539},
    Year = {2006},
      }

    \bib {IshiiQuasi-Gorenstein}{article}{
    	Author = {S. {Ishii}},
    	Title = {{Quasi-Gorenstein Fano 3-folds with isolated non-rational loci.}},
    	FJournal = {{Compositio Mathematica}},
    	Journal = {{Compos. Math.}},
    	ISSN = {0010-437X; 1570-5846/e},
    	Volume = {77},
    	Number = {3},
    	Pages = {335--341},
    	Year = {1991},
    }
     \bib{JohnsonUrlich}{article}{
        Author = {M. {Johnson} and B. {Ulrich}},
        Title = {{Serre's condition $R_{k}$ for associated graded rings}},
        FJournal = {{Proceedings of the American Mathematical Society}},
        Journal = {{Proc. Am. Math. Soc.,}},
        ISSN = {0002-9939; 1088-6826/e},
        Volume = {127,},
        Number = {9},
        Pages = {2619--2624},
        Year = {1999},
      }

      \bib {KunzAlmost}{article}{
      	Author = {E. {Kunz}},
      	Title = {{Almost complete intersections are not Gorenstein rings.}},
      	FJournal = {{Journal of Algebra}},
      	Journal = {{J. Algebra}},
      	ISSN = {0021-8693},
      	Volume = {28},
      	Pages = {111--115},
      	Year = {1974},
      }
      \bib{Marley}{article}{
        Author = {T. {Marley,}  M. W. {Rogers,}  H. {Sakurai}},
    Title = {{Gorenstein rings and irreducible parameter ideals}},
    FJournal = {{Proceedings of the American Mathematical Society}},
    Journal = {{Proc. Am. Math. Soc.,}},
    ISSN = {0002-9939; 1088-6826/e},
    Volume = {136,},
    Number = {1},
    Pages = {49--53},
    Year = {2008},
      }
      \bib{Matsumura}{book}{
       Author = {H. {Matsumura}},
    Title = {{Commutative ring theory. Transl. from the Japanese by M. Reid.}},
    Year = {1986},
    Language = {English},
    HowPublished = {{Cambridge Studies in Advanced Mathematics, 8. Cambridge etc.: Cambridge University Press. XIII, 320 p. {\L} 30.00; {\$} 49.50 (1986).}}
      }
      \bib{Miazaki}{article}{
         AUTHOR = {Miyazaki, C.},
         TITLE = {Graded {B}uchsbaum algebras and {S}egre products},
         JOURNAL = {Tokyo J. Math.,},
         FJOURNAL = {Tokyo Journal of Mathematics},
         VOLUME = {12,},
         YEAR = {1989},
         NUMBER = {1},
         PAGES = {1--20},
      }
      \bib{Nishimura}{article}{
        author={J. Nishimura}
        title={A few examples of local rings, I}
        journal={Kyoto Journal of Mathematics,}
        volume={52,}
        year={2012}
        pages={51–87}
      }
      \bib {OchiaiShimomotoBertini}{article}{
      	Author = {T. {Ochiai} and K. {Shimomoto}},
      	Title = {{Bertini theorem for normality on local rings in mixed characteristic (applications to characteristic ideals).}},
      	FJournal = {{Nagoya Mathematical Journal}},
      	Journal = {{Nagoya Math. J.}},
      	ISSN = {0027-7630; 2152-6842/e},
      	Volume = {218},
      	Pages = {125--173},
      	Year = {2015},
      }
      \bib{O'Carroll}{article}{
        author={L. O'Carroll}
        title={On the generalized fractions of Sharp and Zakeri}
        journal={J. London Math. Soc.,}
        volume={28 (2),}
        year={1983}
        pages={417-427}
      }
      
      \bib{Rotman}{book}{
        Author = {J. J. {Rotman}},
        Title = {{An introduction to homological algebra. 2nd ed.}},
        Edition = {2nd ed.},
        ISBN = {978-0-387-24527-0/pbk; 978-0-387-68324-9/ebook},
        Pages = {xiv + 709},
        Year = {2009},
        Publisher = {Berlin: Springer},
      }
      \bib{Sazeedeh}{article}{
         Author = {R. {Sazeedeh}},
    Title = {{Gorenstein injectivity of the section functor}},
    FJournal = {{Forum Mathematicum}},
    Journal = {{Forum Math.,}},
    ISSN = {0933-7741; 1435-5337/e},
    Volume = {22,},
    Number = {6},
    Pages = {1117--1127},
    Year = {2010},
      }
      \bib{Schenzel}{article}{
        Author = {P. {Schenzel}},
        Title = {{A note on almost complete intersections}},
        
        Journal = {{Seminar D. Eisenbud,  B. Singh and W. Vogel, Vol. 2, Teubner-Texte Math. 48, 49-54 (1982).}},
      }
      \bib{SchenzelLocalCohomology}{book}{
        Author = {Peter {Schenzel}},
         Title = {{On the use of local cohomology in algebra and geometry}},
        BookTitle = {{Six lectures on commutative algebra. Lectures presented at the summer school, Bellaterra, Spain, July 16--26, 1996}},
         ISBN = {3-7643-5951-X/hbk},
         Pages = {241--292},
         Year = {1998},
         Publisher = {Basel: Birkh\"auser},
      }

    \bib{Cousin}{article}{
         Author = {R.Y. {Sharp}},
          Title = {{The Cousin complex for a module over a commutative Noetherian ring}},
        FJournal = {{Mathematische Zeitschrift}},
        Journal = {{Math. Z.,}},
         ISSN = {0025-5874; 1432-1823/e},
        Volume = {112,},
        Pages = {340--356},
        Year = {1969},
      }
    \bib {SinghCyclic}{article}{
        Author = {Anurag K. {Singh}},
        Title = {{Cyclic covers of rings with rational singularities.}},
        FJournal = {{Transactions of the American Mathematical Society}},
        Journal = {{Trans. Am. Math. Soc.}},
        ISSN = {0002-9947; 1088-6850/e},
        Volume = {355},
        Number = {3},
        Pages = {1009--1024},
        Year = {2003},	 
    }
    \bib{Stuckrad}{book}{
      Author = {J. {St\"uckrad} and W. {Vogel}},
    Title = {{Buchsbaum rings and applications. An interaction between algebra, geometry and topology}},
    ISBN = {3-540-16844-3},
    Year = {1986},
    Language = {English},
    }
    
    \bib {TavanfarReduction}{article}{
    	Author={Ehsan Tavanfar},
    	Title={Reduction of certain homological conjectures to excellent unique factorization domains},   	 	
    	Journal= { arXiv:1607.00025v1 [math.AC]},
    }
    
    \bib{UlrichGorenstein}{article}{
    	AUTHOR = {B. {Ulrich}},
    	TITLE = {Gorenstein rings as specializations of unique factorization
    		domains},
    	JOURNAL = {J. Algebra},
    	FJOURNAL = {Journal of Algebra},
    	VOLUME = {86},
    	YEAR = {1984},
    	NUMBER = {1},
    	PAGES = {129--140},
    }

    \bib{Yoshizawa}{article}{
      Author = {T. {Yoshizawa}},
    Title = {{On Gorenstein injectivity of top local cohomology modules}},
    FJournal = {{Proceedings of the American Mathematical Society}},
    Journal = {{Proc. Am. Math. Soc.,}},
    ISSN = {0002-9939; 1088-6826/e},
    Volume = {140,},
    Number = {6},
    Pages = {1897--1907},
    Year = {2012},
    Publisher = {American Mathematical Society, Providence, RI}
    }
    
    
    \bib{Zargar}{article}{
       Author = {M. {Rahro Zargar} and H. {Zakeri}},
      Title={On injective and Gorenstein injective dimensions of local cohomology
modules},
      Journal={arXiv:1204.2394v1 [math.AC] 11},
      Year={2012}
    }

   \end{biblist}
\end{document}